\documentclass[letterpaper,12pt]{conm-p-l}

\usepackage{textcomp} % provides extra symbols
\usepackage{graphics} % provides resizebox command to resize large
\usepackage{amssymb} % provides symbols and fonts
\usepackage{mathtools}
\usepackage{thmtools} % provides flexible reference for thm environments and other cool features
\usepackage{mathrsfs} % provides the command \mathscr
\usepackage{MnSymbol} % provides dotcup and other mathematical symbols
\usepackage{extpfeil} % provides additional arrows
\usepackage{verbatim} % provides multiline comments
\usepackage{longtable} % for tables spanning multiple pages
\usepackage{bigstrut} % needed to add extra vertical space in rows of tabular 
\usepackage[table]{xcolor} % for colouring cells in tables
\usepackage{afterpage} % to place longtables on the next page using \afterpage
\usepackage{xparse} % needed to define commands with more than one optional argument
\usepackage{xfrac} % needed for slanted fractions provided by \sfrac
\usepackage{mathscinet} % needed for \germ and other commands sometimes used
\usepackage{adjustbox} % used to resize tables if they are too large
\usepackage[obeyDraft]{todonotes} % provides todolists and todonotes

\usepackage[final,                  % overwrite possible draft option in documentclass
            pdftex,                 % use pdftex
            pdfpagelabels,
            pdfstartview = {FitH},  % fit the width of the page to the window
            bookmarks,              % create bookmarks for the pdf
            pdfpagemode=UseNone,    % do not show bookmarks when opening the pdf
            colorlinks,             % colour the text of the links instead of using a coloured box
            plainpages = false,     % make page anchors using the formatted form of the page number 
            linktoc=all,            % set to all if you want both sections and subsections linked
            linkcolor=blue,         % colour for the links
            citecolor=red         % colour for the citations
            ]{hyperref}
            
\hypersetup{pdftitle = {Some Examples of the p-Canonical Basis},
            pdfauthor = {Lars Thorge Jensen, Geordie Williamson},
            pdfkeywords = {Hecke Algebra} {Canonical Basis} {Categorification} {Parity Sheaves},
            pdfdisplaydoctitle      % display document title instead of file name in title bar 
            }

\usepackage[arrow, matrix, curve]{xy} % provides commutative diagrams
\usepackage{tikz} % used for digrammatics
\usetikzlibrary{arrows, fit, calc, shapes.misc, shapes.geometric, decorations.markings}
\usepackage{pgfplots}             
            
\allowdisplaybreaks % allow page breaks inside align
            
\declaretheoremstyle[
spaceabove=4pt, spacebelow=4pt,
headfont=\normalfont\bfseries,
notefont=\mdseries, notebraces={(}{)},
bodyfont=\itshape
]{plain}

\declaretheoremstyle[
spaceabove=4pt, spacebelow=4pt,
headfont=\normalfont\bfseries,
notefont=\mdseries, notebraces={(}{)},
bodyfont=\itshape,
postheadspace=\newline
]{break}

\declaretheorem[title=Theorem, style=plain, numberwithin=section]{thm}
\declaretheorem[title=Proposition, style=plain, numberlike=thm]{prop}
\declaretheorem[title=Lemma, style=plain, numberlike=thm]{lem}

\declaretheorem[title=Assumption, style=plain, numberlike=thm]{ass}

\declaretheorem[title=Theorem, style=break, numberlike=thm]{thmlab}

\declaretheorem[title=Definition, style=definition, numberlike=thm]{defn}

\declaretheorem[title=Remark, style=remark, numberlike=thm]{remark}

\usepackage{cleveref} % Includes environment names in references and pays attention
\crefname{thm}{Theorem}{Theorems}
\crefname{prop}{Proposition}{Propositions}
\crefname{lem}{Lemma}{Lemmata}
\crefname{cor}{Corollary}{Corollaries}
\crefname{rem}{Reminder}{Reminders}
\crefname{defn}{Definition}{Definitions}

\crefname{thmlab}{Theorem}{Theorems}
\crefname{proplab}{Proposition}{Propositions}
\crefname{lemlab}{Lemma}{Lemmata}
\crefname{corlab}{Corollary}{Corollaries}
\crefname{remlab}{Reminder}{Reminders}
\crefname{conj}{Conjecture}{Conjectures}

\crefname{thmreflab}{Theorem}{Theorems}
\crefname{propreflab}{Proposition}{Propositions}
\crefname{lemreflab}{Lemma}{Lemmata}
\crefname{correflab}{Corollary}{Corollaries}
\crefname{remreflab}{Reminder}{Reminders}
\crefname{conjref}{Conjecture}{Conjectures}

\crefname{remark}{Remark}{Remarks}
\crefname{claim}{Claim}{Claims}
\crefname{ex}{Example}{Examples}

\crefname{figure}{Figure}{Figures}

\CompileMatrices
\entrymodifiers={+!!<0pt,\fontdimen22\textfont2>} 

\def\clap#1{\hbox to 0pt{\hss#1\hss}}

\makeatletter
\def\underbracket{%
    \@ifnextchar[{\@underbracket}{\@underbracket [\@bracketheight]}%
}
\def\@underbracket[#1]{%
    \@ifnextchar[{\@under@bracket[#1]}{\@under@bracket[#1][0.4em]}%
}
\def\@under@bracket[#1][#2]#3{%\message {Underbracket: #1,#2,#3}
    \mathop{\vtop{\m@th \ialign {##\crcr $\hfil \displaystyle {#3}\hfil $%
    \crcr \noalign {\kern 3\p@ \nointerlineskip }\upbracketfill {#1}{#2}
    \crcr \noalign {\kern 3\p@ }}}}\limits}

\def\upbracketfill#1#2{$\m@th \setbox \z@ \hbox {$\braceld$}
    \edef\@bracketheight{\the\ht\z@}\bracketend{#1}{#2}
    \leaders \vrule \@height #1 \@depth \z@ \hfill
    \leaders \vrule \@height #1 \@depth \z@ \hfill \bracketend{#1}{#2}$}

\def\bracketend#1#2{\vrule \topsepheight #2 width #1\relax}
\makeatother

\makeatletter
\def\thmt@refnamewithcomma #1#2#3,#4,#5\@nil{%
  \@xa\def\csname\thmt@envname #1utorefname\endcsname{#3}%
  \ifcsname #2refname\endcsname
    \csname #2refname\expandafter\endcsname\expandafter{\thmt@envname}{#3}{#4}%
  \fi
}
\makeatother

\makeatletter
\newcommand*\rel@kern[1]{\kern#1\dimexpr\macc@kerna}
\newcommand*\widebar[1]{%
  \begingroup
  \def\mathaccent##1##2{%
    \rel@kern{0.8}%
    \overline{\rel@kern{-0.8}\macc@nucleus\rel@kern{0.2}}%
    \rel@kern{-0.2}%
  }%
  \macc@depth\@ne
  \let\math@bgroup\@empty \let\math@egroup\macc@set@skewchar
  \mathsurround\z@ \frozen@everymath{\mathgroup\macc@group\relax}%
  \macc@set@skewchar\relax
  \let\mathaccentV\macc@nested@a
  \macc@nested@a\relax111{#1}%
  \endgroup
}
\makeatother

\DeclareMathOperator{\Hom}{Hom}
\DeclareMathOperator{\End}{End}

\DeclareMathOperator{\id}{id}

\DeclareMathOperator{\ch}{ch}
\DeclareMathOperator{\Ex}{Ex}
\DeclareMathOperator{\df}{df}
\DeclareMathOperator{\rk}{rk}
\DeclareMathOperator{\grk}{grk}
\DeclareMathOperator{\Rep}{Rep}
\DeclareMathOperator{\ind}{ind}
\DeclareMathOperator{\Spec}{Spec}
\DeclareMathOperator{\supp}{supp}

\newcommand\T{\rule{0pt}{2.6ex}}       % Top strut
\newcommand\B{\rule[-1.2ex]{0pt}{0pt}} % Bottom strut

\newcommand{\defeq}{\ensuremath{\coloneqq}}

\newcommand{\cat}[1]{\ensuremath{\mathcal{#1}}}

\newcommand{\karalg}[1]{\ensuremath{\cat{K}ar(#1)}}

\newcommand{\Galg}[1]{\ensuremath{[#1]}}

\newcommand{\Z}{\ensuremath{\mathbb{Z}}}
\newcommand{\N}{\ensuremath{\mathbb{N}}}
\newcommand{\Q}{\ensuremath{\mathbb{Q}}}
\newcommand{\R}{\ensuremath{\mathbb{R}}}
\newcommand{\C}{\ensuremath{\mathbb{C}}}

\newcommand{\K}{\ensuremath{\mathbb{K}}}
\newcommand{\rO}{\ensuremath{\mathbb{O}}}

\newcommand{\p}[1]{\ensuremath{\prescript{p}{}{#1}}}
\newcommand{\pre}[2]{\ensuremath{\prescript{#1}{}{#2}}}

\newcommand{\heck}[1][]{\ensuremath{\mathcal{H}_{#1}}}

\newcommand{\std}[1]{\ensuremath{H_{#1}}}
\newcommand{\kl}[1]{\ensuremath{\underline{H}_{#1}}}
\newcommand{\pcan}[2][p]{\ensuremath{\prescript{#1}{}{\underline{H}}_{#2}}}
\newcommand{\desc}[1]{\ensuremath{\mathcal{#1}}}

\newcommand{\expr}[1]{\ensuremath{\underline{#1}}}
\newcommand{\mult}[1]{\ensuremath{\expr{#1}_{\bullet}}}

\newcommand{\rt}[1]{\ensuremath{\alpha_{#1}}}
\newcommand{\cort}[1]{\ensuremath{\alpha_{#1}^{\vee}}}
\newcommand{\rts}{\ensuremath{\Phi}}
\newcommand{\corts}{\ensuremath{\Phi^{\vee}}}
\newcommand{\rtbasis}{\ensuremath{\Delta}}
\newcommand{\cortbasis}{\ensuremath{\Delta^{\vee}}}
\newcommand{\charlat}{\ensuremath{X}}
\newcommand{\cocharlat}{\ensuremath{X^{\vee}}}

\newcommand{\affGr}{\ensuremath{\mathcal{G}r_a}}

\newcommand{\affFl}{\ensuremath{\mathcal{F}l_a}}
\newcommand{\Fl}{\ensuremath{\mathcal{F}l}}

\newcommand{\NH}{\ensuremath{\mathcal{NH}}}

\newcommand{\HC}{\ensuremath{\mathbf{H}}}
\newcommand{\BS}{\ensuremath{\mathbf{H}_{BS}}}
\newcommand{\HCnl}[1]{\ensuremath{\HC^{\nless #1}}}
\newcommand{\Homnl}[2][]{
   \ifthenelse{ \equal{#1}{} } { \ensuremath{\Hom_{\nless #2}} } 
   { \ensuremath{\Hom_{\nless #2, #1}} } }
\newcommand{\Endnl}[2][]{
   \ifthenelse{ \equal{#1}{} } { \ensuremath{\End_{\nless #2}} } 
   { \ensuremath{\End_{\nless #2, #1}} } }

\newcommand{\LL}[2]{\ensuremath{\mathrm{LL}_{\expr{#1}, \expr{#2}}}}
\newcommand{\dL}[3]{\ensuremath{\mathbb{LL}_{#1, \expr{#2}, \expr{#3}}}} 

\newcommand{\cols}{\ensuremath{{\color{blue} s}}}
\newcommand{\colt}{\ensuremath{{\color{red} t}}}

\newcommand{\llrrbracket}[1]{% \llrrbracket{..}
  [\![#1]\!]}
  
\newcommand{\llrrparen}[1]{% \llrrparen{..}
  (\!(#1)\!)}

\tikzset{%
  DynNode/.style={circle, inner sep=2pt, draw=black, fill=white},
  Greater/.style={pos=0.65, inner sep=0mm, outer sep=0mm},
  highlight/.style={rectangle,rounded corners,fill=red!15,draw=red,
    fill opacity=0.5,thick},
  bendBelow/.style={bend left=70, looseness=2},
  bendAbove/.style={bend right=70, looseness=2},
  object/.style={circle, fill, inner sep=1.5pt, outer sep=0mm},
  labeling/.style={outer sep=0mm, inner sep=0mm},
  1morph/.style={->, shorten >= 0.5pt, >=stealth'},
  2morph/.style={-implies,double,double equal sign distance,
                 shorten >=2pt, shorten <=3pt},
  spot/.style={color=black, thin, dashed},
  sline/.style={color=blue, line width=2pt},
  tline/.style={color=red, line width=2pt},
  uline/.style={color=green, line width=2pt},
  line/.style={color=#1, line width=2pt},
  line/.default=blue,
  sdot/.style={color=blue, thin, fill},
  tdot/.style={color=red, thin, fill},
  udot/.style={color=green, thin, fill},
  dot/.style={color=#1, thin, fill},
  dot/.default=blue
}

\begin{document}

% First declare some Variables needed for TikZ
\def \xdist {0.4cm}
\def \ydist {0.5cm}
\def \circSize {2.5pt} % Size of dots in Soergel diagrams
\def \armLen {0.4cm} % Length of the arms for units in Soergel diagrams
\def \edgeShift {0.5mm} % Vertical edge shift of multiple edges in Dynkin diagrams
\def \wingLen {0.3cm} % Length of the wings of > or < in non-simply laced Dynkin diagrams

\title{The $p$-canonical basis for Hecke algebras}

%    author one information
\author{Lars Thorge Jensen}
\address{Max Planck Institute for Mathematics, Vivatsgasse 7, 53111 Bonn}
\curraddr{}
\email{ltjensen@mpim-bonn.mpg.de}
\thanks{}

%    author two information
\author{Geordie Williamson}
\address{Max Planck Institute for Mathematics, Vivatsgasse 7, 53111 Bonn}
\curraddr{}
\email{geordw@gmail.com}
\urladdr{http://people.mpim-bonn.mpg.de/geordie/}
\thanks{}

%    The 2010 edition of the Mathematics Subject Classification is
%    now available.  If you are citing a classification from the
%    new scheme, use the following input coding instead.
%\subjclass[2010]{Primary }

\date{\today}

\begin{abstract}
   We describe a positive characteristic analogue of the Kazhdan-Lusztig basis
   of the Hecke algebra of a crystallographic Coxeter system and investigate some 
   of its properties. Using Soergel calculus we describe an algorithm to calculate
   this basis. We outline some known or expected applications in modular 
   representation theory. We conclude by giving several examples.
\end{abstract}

\maketitle

\section{Introduction}
One consequence of categorification is that it often leads to a
``canonical'' basis in the structure being categorified. Canonical
bases have numerous remarkable applications in representation theory
and beyond.

The first examples of canonical bases are the Kazhdan-Lusztig basis
of Hecke algebras and the canonical basis of quantum groups. These
bases have geometric origins, arising as the shadow of a sheaf (an
intersection cohomology sheaf) on a variety. That one should obtain
interesting bases in this way is, in some sense, predicted by
Grothendieck's function-sheaf correspondence.\footnote{Grothendieck's
  correspondence can be seen as an early encouragement to pursue
  categorification, by replacing a function (an object of a set) by a
  sheaf or $D$-module (an object of a category). Despite many more examples
  illustrating the power of this correspondence (e.g. Lusztig's theory of
  character sheaves, the geometric Langlands program, \dots), it still
  seems remarkable that Grothendieck's prophecy really works!
}

More recently, considerable progress has been made via algebraic
approaches to categorification. Here the canonical basis arises as the
shadow of a simple or indecomposable projective module over an
algebra. For example, the canonical basis of the quantum group is
realized in simply laced type as the classes in the Grothendieck group 
of indecomposable projective modules for KLR algebras (see 
\cite{VV}). Similarly, the Kazhdan-Lusztig basis arises as the 
classes of indecomposable Soergel bimodules (see \cite{EW1}). The passage between algebra and
geometry is made by realizing an algebra as an extension algebra of
geometric origin \cite{RoKLR, VV, SLanglands}.

%The canonical basis is usually relevant to representation theory in
%characteristic zero. 
The canonical basis (or rather, the base change matrix between the canonical
basis and a standard basis) usually records multiplicities in categories
of interest in representation theory, which are defined over a field
of characteristic $0$. This is reflected by the choice of coefficients
of characteristic zero (both in the geometric and algebraic
settings). Usually it is only via appeal to powerful geometric theorems
(e.g. \cite{DWeil2,BBD,SaDecompThm,dCMDecompThmSurvey}) that one has an algorithm 
to compute the canonical basis (however see \cite{EW1}).

There are several questions in modular representation theory where certain
multiplicities are expected to be given in terms of a canonical basis. Famous
examples include Lusztig's conjecture for simple rational
modules for algebraic groups \cite{LuConj}, the LLT conjecture for representations
of Hecke algebras at roots of unity \cite{LLT,Ar}, and the James conjecture on
representations of the symmetric group \cite{JaConj}. %Here one expects a
%(hopefully explicit) finite set of ``bad'' primes, above which the answer is
%given by the canonical basis (``as in characteristic zero''). 
These conjectures predict that the situation in characteristic $0$ agrees with
the one in characteristic $p \gg 0$, sometimes state an explicit bound
on the ``bad'' primes to exclude, but usually aren't so bold as to make
any statement about what happens for small $p$.

However recent work by several authors suggests a different way of thinking
about this problem. Namely, in important examples there
exists a ``$p$-canonical basis'' for each choice of prime number $p$.
The $p$-canonical basis has the same indexing set as the usual canonical basis, 
and each of its elements agrees with the corresponding element of the 
canonical basis for large enough $p$, but may differ from it for small $p$. One 
hopes that the $p$-canonical basis gives the correct answer in  representation
theory in characteristic $p$. Examples of this phenomenon include:
\begin{enumerate}
\item The work of Grojnowski, Ariki and others \cite{GrojAffineSLp, Ar} 
  identifying the Grothendieck group of the category of representations of 
  all symmetric groups with the basic representation of an affine Lie algebra. 
  Here the $p$-canonical basis is defined via the classes of indecomposable
  projective modules, and hence the $p$-canonical basis contains deep
  information in modular representation essentially by definition. (It
  is from \cite{GrojAffineSLp} that we first learnt the term $p$-canonical basis.)
\item The work of Soergel \cite{SIntCohPosChar}, connecting certain multiplicities in
  the rational representation theory of algebraic groups with
  indecomposable summands of Bott-Samelson sheaves. Subsequent work (e.g. \cite{JMW1, BW})
  shows that Soergel's results may be restated as giving these
  multiplicities in terms of coefficients of the $p$-canonical basis of the
  Hecke algebra of the finite Weyl group.
\item A recent conjecture of Riche and the second author, which
  predicts that the characters of indecomposable tilting modules for
  algebraic groups should be given in terms of coefficients of the
  $p$-canonical basis in the anti-spherical module of the affine Weyl
  group. By Schur-Weyl duality this conjecture implies that all
  decomposition numbers for symmetric groups occur as coefficients in
  this basis. In contrast to (1), this equivalence is not ``by
  definition'', and might lead to more efficient algorithms. In recent work, Riche
  and the second author have confirmed their conjecture for
  the general linear group.
\end{enumerate}

The above approach breaks problems in modular representation
theory into two subproblems. First one should connect the problem in
representation theory to the $p$-canonical basis, and then one should
calculate the $p$-canonical basis. In general the second step seems to
be extremely difficult, and it is possibly too optimistic to expect an
answer in general \cite{WTorsion}. However at the very least this approach has 
the potential to unify questions in modular representation theory (just as
different questions may have answers given by the same Kazhdan-Lusztig 
polynomials). Also, there are combinatorial constraints on the $p$-canonical 
basis which make the situation very rigid, and might merit further
investigation.

\subsection{}
In this paper we recall the definition of and study the $p$-canonical basis of the Hecke
algebra of a crystallographic Coxeter system $(W,S)$ (see \cite{WParityOberwolfach} where
the definition appears in print for the first time). To motivate its
definition, we first recall how the Kazhdan-Lusztig basis arises from
categorification. In the introduction, let us assume that $W$ is the Weyl
group of a complex reductive group $G$ with
 a maximal torus $T$. Let $B \subset G$ denote the Borel subgroup
 corresponding to the simple reflections $S \subset W$.

The Hecke algebra $\heck$ of $(W,S)$ has two (essentially equivalent)
categorifications, often referred to loosely as the Hecke category:
\begin{enumerate}
\item \emph{Geometric:} The additive, monoidal (under convolution) category of 
  $B$-biequivariant semi-simple complexes on $G$: the full subcategory of the 
  equivariant derived category $D^b_{B \times B}(G, \mathbb{Q})$ consisting of 
  direct sums of shifts of equivariant intersection cohomology complexes.
\item \emph{Algebraic:} The monoidal category of Soergel bimodules: a
  certain full subcategory of the monoidal category of graded $R$-bimodules, 
  where $R$ denotes the regular functions on the Lie algebra of the maximal torus
  $T$.
\end{enumerate}
In the first setting, the Kazhdan-Lusztig basis arises as the graded
dimensions of stalks of the intersection cohomology complexes (see \cite{KL2}). 
In the second case, the Kazhdan-Lusztig basis is realized as the 
characters of the indecomposable self-dual Soergel bimodules (see 
\cite{S3, EW1}).

In \cite{EW2} the monoidal category of Soergel bimodules is described
by generators and relations, following earlier work by Elias \cite{E3},
Elias-Khovanov \cite{EKh} and Libedinsky \cite{Li3} (we recall this description 
in detail below). The upshot is that there exists a graded
monoidal category $\HC$ which is defined over the integers, and
whose extension of scalars to $\mathbb{Q}$ is equivalent to Soergel
bimodules. Hence one can think of $\HC$ as an integral form of
the Hecke category. For any field $k$ we can consider the
extension of scalars $\pre{k}{\HC}$ to $k$ and it is proved in \cite{EW2}
that one has an canonical ``character'' isomorphism of
$\mathbb{Z}[v,v^{-1}]$-algebras
\[
\ch : [\pre{k}{\HC}] \stackrel{\sim}{\longrightarrow} \heck
\]
between the split Grothendieck group of $\pre{k}{\HC}$ and the Hecke
algebra. Hence for any field $\pre{k}{\HC}$ provides a
categorification of the Hecke algebra.\footnote{There is a
minor additional technical assumption if the characteristic of $k$
is 2. We ignore this point in the introduction.} (Note that while
the coefficients of $\HC$ change, the Grothendieck group
is always the same Hecke algebra over $\mathbb{Z}[v, v^{-1}]$.)

In \cite{EW2} the indecomposable objects of $\pre{k}{\HC}$ are
classified, following Soergel's classification of the indecomposable
Soergel bimodules in \cite{S4}. It turns out that for all $w \in W$ there exists an
indecomposable object $\pre{k}{B}_w \in \pre{k}{\HC}$, and that any
indecomposable object is isomorphic to a grading shift of $\pre{k}{B}_w$ 
for some $w \in W$. The $p$-canonical basis is defined as the character of this
indecomposable object:
\[
\pcan{w} := \ch(\pre{k}{B}_w)
\]
where $p$ denotes the characteristic of $k$.
From basic properties of $\pre{k}{\HC}$ it is easy to see that $\{
\pcan{w} \; | \; w \in W \}$ is a basis for the Hecke algebra which
only depends on the characteristic of $k$, and
that its structure constants are positive (see \Cref{propPCanProps}). Moreover,
because $\pre{\Q}{\HC}$ is equivalent to Soergel bimodules with
$\mathbb{Q}$-coefficients, $\{\pcan[0]{w} \; | \; w \in W \}$ is the
Kazhdan-Lusztig basis.

\subsection{}

In this subsection we outline the connection between the $p$-canonical basis and 
parity sheaves on (affine) flag varieties. The reader unfamiliar with affine 
flag varieties may keep the important case of a (finite) flag variety in mind. 
%or take a look at \cite{KuKMGrps} as standard reference.

To any based root datum we can associate a connected, reductive algebraic group
scheme $G$ over $\Z$ with Borel subgroup $B \subseteq G$ and  maximal torus 
$T \subseteq B \subseteq G$. In the finite case, let $X$ denote the $\C$-points
of the flag variety $\Fl \defeq G / B$ with its classical metric topology. 

In the affine setting, define the loop group $LG$ (resp. positive loop group $L^+ G$) of $G$ 
as the $\Z$-functor given by $R \mapsto G(R\llrrparen{t})$ (resp. $R \mapsto G(R\llrrbracket{t})$). 
Denote by $I$ the Iwahori subgroup given by the inverse image of $B$ under the
morphism $L^+ G \rightarrow G$ induced by $t \mapsto 0$. In this case we define $X$
to be the $\C$-points of the affine flag variety $\affFl \defeq LG / I$ viewed 
as ind-variety (see \cite{Goer} for more information).

In both settings we have an (Iwahori)-Bruhat decomposition expressing the 
corresponding flag variety as a disjoint union of left $B(\C)$ (resp. 
$I( \C )$)-orbits indexed by the (extended affine) Weyl group. The closure
relation is given by the Bruhat order. Note that the affine flag variety is 
isomorphic as an ind-variety to a suitable disjoint union of Kac-Moody flag 
varieties (see e.g. \cite{KuKMGrps}).

Fix a field $k$ of coefficients. We may consider $D^b_{H}(X)$ the $H$-equivariant 
bounded derived category of $k$-sheaves where $H$ is either $B(\C)$ or 
$I(\C)$ depending on the setting (see \cite{BL} for more information about 
equivariant derived categories).
In \cite[\S 4.1]{JMW1} Juteau, Mautner and the second author introduce and prove 
the existence of parity sheaves on generalized flag varieties, a class of 
objects in $D^b_{H}(X)$ whose stalks satisfy a cohomological parity 
vanishing condition (for the trivial pariversity function). Their work was 
motivated by Soergel's idea to consider another class of objects as 
``replacements'' for intersection cohomology complexes with positive 
characteristic coefficients (see \cite{SIntCohPosChar}). Observe that while $X$ is 
still a variety (resp. an ind-variety) over $\C$ equipped with its classical 
topology, the coefficients of the sheaves we are studying may lie in a field of 
positive characteristic. 

The theory of parity sheaves parallels the theory of perverse sheaves in the 
following points:
\begin{enumerate}
   \item indecomposable parity sheaves are classified analogously to simple perverse
         sheaves, being up to isomorphism the unique extension of an irreducible 
         local system on a stratum;
         %(resp. indecomposable for coefficients not in a field) 
   \item in our setting, if the coefficients are a field of characteristic zero 
         the intersection cohomology sheaves are parity sheaves. 
\end{enumerate}
But while the decomposition theorem for perverse sheaves fails in positive
characteristic, the pushforward along a proper, even, stratified map preserves
parity sheaves and it is possible to calculate the multiplicities
of the occurring indecomposable parity sheaves via intersection forms (see
\cite[\S 3.3.]{JMW1}). Thus parity sheaves are particularly interesting in the 
case of positive characteristic coefficients.

Parity sheaves on various varieties have also been used for categorification.
In \cite{Mak} Maksimau realizes Lusztig's integral form of the positive half of the 
quantum group associated to a Dynkin quiver as a coalgebra geometrically, by 
considering parity sheaves on quiver moduli spaces. In our setting, 
parity sheaves also give canonical bases of the Hecke algebra. If the 
coefficients are a field of characteristic zero, then the graded dimensions  
of stalks of parity sheaves give the Kazhdan-Lusztig basis
(as mentioned above, see \cite{KL2, SpIntCoh}). For a field of positive 
characteristic the indecomposable parity sheaves realize the $p$-canonical basis
in this way (see \cite{BW} in the setting of the flag variety)\footnote{In many
settings this statement follows from \cite{S4,FW,ARi}. In the generality discussed
in this paper this result will appear in \cite{RW}.}.

From this perspective, analyzing when the $p$-canonical basis differs from
the Kazhdan-Lusztig basis has several interesting geometric implications. First 
of all, it allows one to study the failure of the decomposition theorem in the 
modular setting (see \cite[\S3]{JMW1}). Secondly, there are close connections 
between the decomposition matrix for intersection cohomology complexes and the 
base change matrix between the Kazhdan-Lusztig and the $p$-canonical basis. In 
\cite[Theorem 2.6]{ARModPervShI} Achar and Riche show that the base change 
matrix gives a $q$-refinement of the decomposition matrix on the Langlands 
dual flag variety. Moreover, in \cite{WRedCharVarTypeA} the second author 
proves that certain base change coefficients and decomposition numbers coincide 
and uses this to give an example of a reducible characteristic variety in 
type $A$. Non-trivial decomposition numbers for an intersection cohomology 
complex can only occur when the characteristic variety of the corresponding 
regular holonomic $D_X$-module is reducible (see \cite{VW}).

\subsection{}
This article gives a survey on known results about the $p$-canonical basis
and its main purpose is to discuss an algorithm to calculate
the $p$-canonical basis. This algorithm relies heavily on the
diagrammatic description of the Hecke category $\HC$. Motivated by 
de Cataldo and Migliorini's proof of the decomposition theorem (see \cite{dCMDecompThmSurvey}),
the second author showed in \cite[\S3.3]{JMW1} together with Juteau and Mautner
that certain intersection forms govern the decomposition behaviour of
the pushforward of the constant sheaf along a surjective, stratified,
even resolution of singularities. In the diagrammatic framework this 
algorithm becomes more feasible, thanks to an explicit description 
of bases of $\Hom$ spaces. This leads to the algorithm described in this paper.
%The algorithm works as follows: 
%The $p$-canonical basis $\{ \pcan{w} \; \vert \; w \in W \}$ is calculated
%by induction on the Bruhat order starting with the smallest element.
%In order to calculate $\pcan{w}$ one calculates for all $v < w$ the graded
%rank of a certain intersection form to determine the graded multiplicity
%with which the indecomposable parity sheaf on the Schubert cell
%corresponding to $v$ occurs in the pushforward of the constant
%sheaf on the Bott-Samelson resolution of the Schubert variety corresponding
%to $w$. Using the knowledge of $\{ \pcan{v} \; \vert \; v < w\}$ one can 
%invert the resulting matrix of coefficients to obtain $\pcan{w}$.
We conclude the paper by giving many examples to show that the $p$-canonical 
basis may depend in a subtle way on $p$.

\subsection{Structure of the Paper:}

\begin{description}
   \item[\Cref{secBack}] We introduce notation and recall important results about
         the Hecke algebra and Soergel calculus.
   \item[\Cref{secIntForm}] After recalling the definition of the $p$-canonical 
         basis, we explain how to calculate it using intersection forms.
   \item[\Cref{secProp}] The elementary properties of the $p$-canonical basis are
         stated and proved.
   \item[\Cref{secEx}] We give several new and interesting examples of the $p$-ca\-non\-ical 
         basis.
\end{description}

\subsection{Acknowledgements}
We would like to thank Ben Elias for very valuable comments and some minor corrections.
We would also like to thank Nicolas Libedinsky for his feedback and the Max Planck Institute
of Mathematics for financial support.

\section{Background}
\label{secBack}

\subsection{Coxeter Systems and Based Root Data}
\label{secCox}
Let $S$ be a finite set and $(m_{s, t})_{s, t \in S}$ be a matrix with entries in
$\N \cup \{\infty\}$ such that $m_{s,s} = 1$ and $ m_{s, t} = m_{t, s} \geqslant 2$ 
for all $s \neq t \in S$. Denote by $W$ the group generated by $S$ subject to the relations 
$(st)^{m_{s, t}} = 1$ for $s, t \in S$ with $m_{s, t} < \infty$. We say that $(W, S)$
is a \emph{Coxeter system} and $W$ a \emph{Coxeter group}. The Coxeter group $W$
comes equipped with a length function $l: W \rightarrow \N$ and the Bruhat order
$\leqslant$ (see \cite{Hum} for more details). A Coxeter system $(W, S)$ is 
called \emph{crystallographic} if $m_{s, t} \in \{2, 3, 4, 6, \infty\}$ for 
all $s, t \in S$.

Define an \emph{expression} to be a finite sequence of elements in $S$. We denote by
\[ \Ex(S) \defeq \{\varnothing\} \cup \bigcup_{i \in \N \setminus \{0\}} \underbrace{S \times 
   \dots \times S}_{i\text{-times}} \]
the set of all expressions in $S$. For an expression $\expr{w} = (s_1, s_2, \dots, s_n)$
denote its \emph{length} by $l(\expr{w}) = n$. The multiplication gives a canonical map
$\Ex(S) \rightarrow W$, $\expr{w} \mapsto \mult{w}$. An expression $\expr{w}$ in $S$
is called \emph{reduced} if $l(\expr{w}) = l(\mult{w})$. For an expression 
$\expr{w} = (s_1, s_2, \dots, s_n)$ in $S$ a \emph{subexpression} of $\expr{w}$ is
a sequence $\expr{w}^{\expr{e}} = (s_1^{e_1}, s_2^{e_2}, \dots s_n^{e_n})$ 
where $e_i \in \{0, 1\}$ for all $i$. The sequence $\expr{e} = (e_1, e_2, \dots, e_k)$ 
is called the \emph{associated $01$-sequence}. We usually decorate $\expr{e}$ as follows: 
For $1 \leqslant k \leqslant n$ denote by $\expr{w}_{\leqslant k} \defeq (s_1, s_2, \dots, s_k)$ 
the first $k$ terms and set $w_k \defeq 
(\expr{w}_{\leqslant k}^{\expr{e}_{\leqslant k}})_{\bullet}$. Assign to $e_i$ a 
decoration $d_i \in \{U, D\}$ where $U$ stands for \emph{Up} and $D$ for 
\emph{Down} as follows:
\[ d_i \defeq 
   \begin{cases}
      U & \text{if } w_{i-1} s_i > w_{i-1} \text{,} \\
      D & \text{if } w_{i-1} s_i < w_{i-1} \text{.}
   \end{cases}
\]
We often write the decorated sequence as $(d_1 e_1, d_2 e_2, \dots, d_n e_n)$. 
The \emph{defect} of $\expr{e}$ is defined to be
\[ \df(e) \defeq \vert \{ i \; \vert \; d_i e_i = U0 \}\vert - 
            \vert \{ i \; \vert \; d_i e_i = D0 \} \vert \text{.}
\]
To illustrate the definitions, consider for example the case $S = \{s, t\}$ 
and $m_{s,t} = m_{t, s} = 3$ (i.e. type $A_2$). The reduced expression $(s, t, s)$
admits two decorated $01$-sequences expressing $s$:
\begin{align*}
   &(U1, U0, D0) \quad \text{of defect } 0\\
   &(U0, U0, U1) \quad \text{of defect } 2\\
\end{align*}

Recall that given an abstract root datum  $\Psi = (\charlat, \rts, \cocharlat, \corts)$ 
and a basis $\rtbasis \subseteq \rts$ the quadruple 
$\Psi_0 = (\charlat, \rtbasis, \cocharlat, \cortbasis)$ is called a \emph{based
root datum} where $\cortbasis$ is the set of simple coroots
(see \cite[\S 7.4]{SpLinAlgGrps} for the definition of a root datum).
From now on, fix a based root datum $\Psi_0$. The matrix $(\langle \rt{s}, \cort{t} 
\rangle)_{s, t \in S}$ is called the \emph{Cartan matrix} associated to this 
based root datum. By the existence theorem (see \cite[Expos\'e XXV, Th\'eor\`eme 1.1]{SGA3}), 
starting from $\Psi_0$ we get $G$, a split connected reductive
algebraic group scheme over $\Z$, together with a Borel subgroup $B \subseteq G$ 
and a maximal torus $T \subseteq B \subseteq G$ such that the root datum 
determined by $(G, T)$ and the basis given by the simple roots whose root 
groups are contained in $B$ give the corresponding based root datum $\Psi_0$. 
For a root $\alpha \in \rts$ we define the corresponding reflection on 
$\charlat$ via 
\[ s_{\alpha} \lambda \defeq \lambda - \langle \lambda, \alpha^{\vee} \rangle \alpha
   \qquad \text{for all } \lambda \in X \text{.} \]
We denote the Weyl group of $\rts$ by $W \defeq \langle s_{\alpha} \; \vert \; 
\alpha \in \rts \rangle$. %Observe that $N_G(T)/T$ is isomorphic to the finite
%group scheme represented by $\Z W$.

For $S \defeq \{ s_{\alpha} \; \vert \; \alpha \in \rtbasis \}$ the triple 
$(\cocharlat, \rtbasis, \cortbasis)$ gives a (not necessarily symmetric) faithful 
realization of the Coxeter system $(W, S)$ over $\Z$ (as defined in the 
appendix of \cite{E3}). Fix a commutative ring $k$. Then ${}^k V \defeq 
\cocharlat \otimes_{\Z} k$ yields a (potentially non-faithful) realization of 
$(W, S)$ over $k$. Set ${}^k V^{\ast} \defeq \Hom_{k}({}^k V, k)$ and note 
that ${}^k V^{\ast}$ is isomorphic to $\charlat \otimes_{\Z} k$. Throughout we 
will assume the following:
% The realization consists of a  free, finite rank $k$-module $\mathfrak{h}$, 
% together with subsets $\{ \cort{s} \; \vert \; s \in S \}  \subset \mathfrak{h}$ and 
% $\{ \rt{s} \; \vert \; s \in S \} \subset \mathfrak{h}^{\ast} = \Hom_{k}(\mathfrak{h}, k)$ 
% called \emph{simple co-roots} and \emph{simple roots} respectively such that 
% $\langle \rt{s}, \cort{s} \rangle = 2$ for all $s \in S$ and $W$ acts on 
% $\mathfrak{h}$ via 
% \[ s(v) = v - \langle\rt{s}, v \rangle \cort{s}\]
% for $s \in S$ and $v \in \mathfrak{h}$. The matrix 
% $(\langle \rt{s}, \cort{t} \rangle)_{s, t \in S}$ is the \emph{Cartan matrix}
% associated to this realization. 
% 
% Note that \emph{minimal} realizations (i.e. those for which the set 
% $\{ \cort{s} \; \vert \; s \in S \}$ forms a basis of $\mathfrak{h}$)
% can be reconstructed from its Cartan matrix.
% 
% For us the main examples will be crystallographic Coxeter systems for which the 
% geometric representation can be defined over $\Z$ and extended to $k$ to give 
% a realization of $(W, S)$ over any commutative ring $k$. For the finite Weyl groups 
% the geometric representation can be reconstructed from the corresponding Cartan 
% matrix and we can equivalently use a Cartan matrix as input datum.
% 
% \noindent
% Throughout we will assume our realization to satisfy:
\begin{ass}[Demazure Surjectivity]
   The maps $\rt{s}: {}^k V \rightarrow k$ and $\cort{s}: {}^k V^{\ast} 
   \rightarrow k$ are surjective for all $s \in S$.
\end{ass}
\noindent This is automatically satisfied if $2$ is invertible in $k$ or if the
Coxeter system $(W, S)$ is of simply-laced type and of rank $\lvert S \rvert \geqslant 2$.
% Note that one can always enlarge ${}^k V$ to ensure Demazure surjectivity 
% by suitably introducing new simple coroots.

We denote by $R = S({}^k V^{\ast})$ the symmetric algebra of ${}^k V^{\ast}$
over $k$ and view it as a graded ring with ${}^k V^{\ast}$ in degree $2$. The action
of $W$ on ${}^k V$ induces an action on $R$ by functoriality. For any $s \in S$ 
we define the \emph{Demazure operator} $\partial_s: R \rightarrow R(-2)$ via
\[ \partial_s(f) \defeq \frac{f - s(f)}{\rt{s}} \]
where $(1)$ denotes the grading shift down by one: Given a graded $R$-bimodule
$B = \bigoplus_{i\in \Z} B^i$, we denote by $B(1)$ the shifted bimodule with
$B(1)^i = B^{i+1}$. Observe that $\partial_s$ is a well-defined graded 
$R^s$-bimodule homomorphism (see \cite[\S 3.3]{EW2} for more details).

\subsection{The Hecke Algebra}
The Hecke algebra is the free $\Z[v, v^{-1}]$-algebra with $\{ \std{w} \; \vert \; 
w \in W \}$ as basis and multiplication determined by:
\begin{alignat*}{2}
   \std{s}^2 &= (v^{-1} - v) \std{s} + 1  \qquad && \text{for all } s \in S \text{,} \\
   %\underbrace{\std{s} \std{t} \std{s} \dots}_{m_{s, t} \text{ terms}} &= 
   %   \underbrace{\std{t} \std{s} \std{t} \dots}_{m_{s,t} \text{ terms}} 
   %   && \text{for all } s \neq t \in S \text{.}
   \std{x} \std{y} &= \std{xy} && \text{if } l(x) + l(y) = l(xy) \text{.}
\end{alignat*}

There is a unique $\Z$-linear involution $\widebar{(-)}$ on $\heck$ satisfying
$\widebar{v} = v^{-1}$ and $\widebar{\std{x}} = \std{x^{-1}}^{-1}$. The Kazhdan-Lusztig 
basis element $\kl{x}$ is the unique element in $\std{x} + \sum_{y < x} v\Z[v] H_y$ 
that is invariant under $\widebar{(-)}$. This is Soergel's normalization from 
\cite{S2} of a basis introduced originally in \cite{KL}.

\subsection{Soergel Calculus}

We define an \emph{$S$-graph} to be a finite, decorated, planar graph with boundary
properly embedded in the planar strip $\R \times [0,1]$ whose edges are coloured by $S$
and all of whose vertices are of the following types:
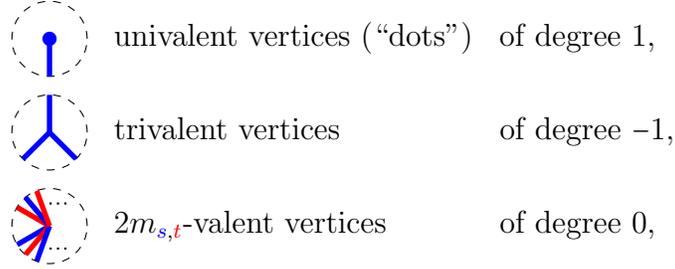
\begin{center}
   \begin{tabular}{c l l}
      %\caption{Vertices in an $S$-graph}
      \label{tabGenMorph} \\
      \centering
      \begin{tikzpicture}[baseline=-0.5ex]
         \draw[spot] (0,0) circle (0.5cm);
         \draw[sline] (0,-0.5) -- (0,0);
         \draw[sdot] (0,0) circle (\circSize);
      \end{tikzpicture} & univalent vertices (``dots'') & of degree $1$, \\[0.5cm]
      \begin{tikzpicture}[baseline=-0.5ex]
         \draw[spot] (0,0) circle (0.5cm);
         \clip (0,0) circle (0.5cm);
         \draw[sline] (-1,-1) -- (0,0);
         \draw[sline] (1,-1) -- (0,0);
         \draw[sline] (0,0) -- (0,1);
      \end{tikzpicture} & trivalent vertices & of degree $-1$, \\[0.5cm]
   %    \begin{tikzpicture}[baseline=-0.5ex]
   %       \draw[spot] (0,0) circle (0.5cm);
   %       \clip (0,0) circle (0.5cm);
   %       \node (f) at (0,0) {$f$};
   %    \end{tikzpicture} & decorating box & of degree $\deg(f)$ \\[0.5cm]
      \begin{tikzpicture}[baseline=-0.5ex, outer sep=0.1mm, inner sep=0mm]
         \draw[spot] (0,0) circle (0.5cm);
         \clip (0,0) circle (0.5cm);         
         \draw[sline] (0,0) -- (210:0.5cm);
         \draw[tline] (0,0) -- (230:0.5cm);
         \draw[sline] (0,0) to (250:0.5cm);
         
         \draw[tline] (0,0) -- (150:0.5cm);
         \draw[sline] (0,0) -- (130:0.5cm);
         \draw[tline] (0,0) to (110:0.5cm);
         
         \draw[dotted, thick] (0, -0.3) to (0.25, -0.3);
         \draw[dotted, thick] (0, 0.3) to (0.25, 0.3);
         
      \end{tikzpicture} & $2m_{\cols, \colt}$-valent vertices & of degree $0$, \\[0.5cm]
   \end{tabular} \\
\end{center}
\noindent
where we require the $2m_{\cols, \colt}$-valent vertex to have exactly
$2m_{s, t}$ edges, coloured alternately by $\cols$ and $\colt$ around
the vertex.

The regions of an $S$-graph (i.e. the connected components of the complement of 
the graph in $\R \times [0,1]$) may be \emph{decorated} by homogeneous elements of $R$. 
The degree of a decorated $S$-graph is defined as the sum of the degrees
of its vertices and of the degrees of the polynomials decorating its regions.

Next, we introduce the diagrammatic category of Soergel bimodules. The main reference
for this is \cite{EW2} (see also \cite{E3} in the dihedral case and \cite{EKh} 
in type $A$).

Let $\BS$ be the strict monoidal category with $\Z$-graded $\Hom$-spaces which 
is monoidally generated by the elements in $S$. Thus the objects of $\BS$ are
given by $\Ex(S)$ and the monoidal structure on the level of objects is given by
concatenation of sequences in $S$. For any $\expr{x}, \expr{y} \in \Ex(S)$,
$\Hom_{\BS}(\expr{x}, \expr{y})$ is defined to be the free $R$-module generated
by isotopy classes of decorated $S$-graphs with bottom boundary $\expr{x}$ and top
boundary $\expr{y}$ modulo the local relations below. The composition (resp. 
tensor product) of two morphisms is given by vertical (resp. horizontal) 
concatenation of diagrams.

We now recall the relations defining $\BS$:

\subsubsection{One-colour Relations} For all $\cols \in S$ we have:

\begin{itemize}
   \item Frobenius Unit:
\begin{equation}
   \label{eqUnit}
   \begin{tikzpicture}[baseline=(current  bounding  box.center), 
                       inner sep=0mm, outer sep=0mm]      
      \draw[spot] (0,0) circle (0.75cm);
      \draw[sline] (0,-0.75) to (0,0.75);
      \draw[sline] (0,0) to (0.3,0);
      \draw[sdot] (0.3,0) circle (\circSize);
   \end{tikzpicture} 
   \quad = \quad
   \begin{tikzpicture}[baseline=(current  bounding  box.center), 
                       inner sep=0mm, outer sep=0mm]      
      \draw[spot] (0,0) circle (0.75cm);
      \draw[sline] (0,-0.75) to (0,0.75);
   \end{tikzpicture}
\end{equation}

   \item Frobenius Associativity:
\begin{equation}
   \label{eqAss}
   \begin{tikzpicture}[baseline=(current  bounding  box.center), 
                       inner sep=0mm, outer sep=0mm]
      
      \draw[spot] (0,0) circle (0.75cm);      
      \draw[sline] (-45:0.75cm) -- (0.25,0) -- (45:0.75cm);
      \draw[sline] (135:0.75cm) -- (-0.25,0) -- (-135:0.75cm);
      \draw[sline] (-0.25,0) to (0.25,0);
   \end{tikzpicture}
   \quad = \quad
   \begin{tikzpicture}[baseline=(current  bounding  box.center), 
                    inner sep=0mm, outer sep=0mm, rotate=90]
      
      \draw[spot] (0,0) circle (0.75cm);      
      \draw[sline] (-45:0.75cm) -- (0.25,0) -- (45:0.75cm);
      \draw[sline] (135:0.75cm) -- (-0.25,0) -- (-135:0.75cm);
      \draw[sline] (-0.25,0) to (0.25,0);
   \end{tikzpicture}
\end{equation}

   \item Needle Relation:
\begin{equation}
   \label{eqNeedle}
   \begin{tikzpicture}[baseline=(current  bounding  box.center), 
                       inner sep=0mm, outer sep=0mm]      
      \draw[spot] (0,0) circle (0.75cm);
      \draw[sline] (0,0) circle (0.25cm);
      \draw[sline] (0,-0.75) to (0,-0.25);
      \draw[sline] (0,0.25) to (0,0.75);
   \end{tikzpicture} 
   \quad = \quad 0
\end{equation}

   \item Barbell Relation:
\begin{equation}
   \label{eqAlpha}
   \begin{tikzpicture}[baseline=(current  bounding  box.center), 
                       inner sep=0mm, outer sep=0mm]      
      \draw[spot] (0,0) circle (0.75cm);
      \draw[sline] (0,-0.3) to (0,0.3);
      \draw[sdot] (0,0.3) circle (\circSize);
      \draw[sdot] (0,-0.3) circle (\circSize);
   \end{tikzpicture} \quad = \quad
   \begin{tikzpicture}[baseline=(current  bounding  box.center), 
                       inner sep=0mm, outer sep=0mm]  
      \draw[spot] (0,0) circle (0.75cm);
      \node (v) at (0,0) {$\rt{\cols}$};
   \end{tikzpicture}
\end{equation}

   \item Nil Hecke Relation:
\begin{equation}
   \label{eqPolSlide}
   \begin{tikzpicture}[baseline=(current  bounding  box.center), 
                       inner sep=0mm, outer sep=1mm]  
      \draw[spot] (0,0) circle (0.75cm);
      \clip (0,0) circle (0.75cm);
      \draw[sline] (0,-1) to node[left, color=black] (f) {\small$f$} (0,1);
   \end{tikzpicture} \quad = \quad
   \begin{tikzpicture}[baseline=(current  bounding  box.center), 
                       inner sep=0mm, outer sep=1mm]  
      \draw[spot] (0,0) circle (0.75cm);
      \clip (0,0) circle (0.75cm);
      \draw[sline] (0,-1) to node[right, color=black] (sf) {\small$\cols f$} (0,1);
   \end{tikzpicture} \quad + \quad
   \begin{tikzpicture}[baseline=(current  bounding  box.center), 
                       inner sep=0mm, outer sep=1mm]  
      \draw[spot] (0,0) circle (0.75cm);
      \clip (0,0) circle (0.75cm);
      \draw[sline] (0,-1) to (0,-0.4);
      \draw[sdot] (0,-0.4) circle (\circSize);
      \draw[sline] (0,1) to (0,0.4);
      \draw[sdot] (0,0.4) circle (\circSize);
      \node (parsf) at (0,0) {\small$\partial_{\cols}(f)$};
   \end{tikzpicture}
\end{equation}
\end{itemize}

\subsubsection{Two-colour Relations} There are two colour relations
for all pairs $\cols, \colt \in S$ such that $m_{\cols, \colt} <
\infty$ (so that the $2m_{st}$-valent vertex is defined).

The first two-colour relation is called \emph{Two-colour Associativity} and describes
what happens what happens when we pull a trivalent vertex through a 
$2m_{\cols, \colt}$-valent vertex. We give it for $m_{\cols, \colt} \in \{2, 3, 4\}$
and let the reader guess the general form (see \cite[(6.12)]{E3}):
\begin{alignat*}{2}
   %\label{eqTrivalentPull2}
   \begin{tikzpicture}[baseline=(current  bounding  box.center), scale=0.75]
      \draw[spot] (0,0) circle (1cm);     
      \draw[tline] (-45:1) -- (135:1);
      \draw[sline] (-135:1) -- (45:0.3);
      \draw[sline] (45:0.3) -- (20:1);
      \draw[sline] (45:0.3) -- (80:1);
   \end{tikzpicture} 
   \enspace &= \enspace
   \begin{tikzpicture}[baseline=(current  bounding  box.center), scale=0.75]
      \draw[spot] (0,0) circle (1cm);
      \draw[tline] (-45:1) -- (135:1);
      \draw[sline] (-135:1) -- (-135:0.5);
      \draw[sline] (-135:0.5) -- (20:1);
      \draw[sline] (-135:0.5) -- (80:1); 
   \end{tikzpicture} 
   \qquad && \text{if  } m_{\cols, \colt} = 2 \quad (\text{type } A_1 \times A_1) \text{,} \\[0.5cm]
   %\label{eqTrivalentPull3}
   \begin{tikzpicture}[baseline=(current  bounding  box.center), scale=0.75]
      \draw[spot] (0,0) circle (1cm);
      \coordinate (a1) at (30:1);
      \coordinate (a2) at (60:1);
      \coordinate (a3) at (90:1);
      \coordinate (a4) at (135:1);
      \coordinate (a5) at (-135:1);
      \coordinate
(a6) at (-90:1);
      \coordinate (a7) at (-45:1);
      \coordinate (l1) at (45:0.4); 
      \coordinate (l0) at (0:0);
      \draw[tline] (a5) -- (l0) -- (a3);
      \draw[tline] (l0) -- (a7);
      \draw[sline] (a1) -- (l1) -- (a2);
      \draw[sline] (l1) -- (l0) -- (a4);
      \draw[sline] (l0) -- (a6);
   \end{tikzpicture} 
   \enspace &= \enspace
   \begin{tikzpicture}[baseline=(current  bounding  box.center), scale=0.75]
      \draw[spot] (0,0) circle (1cm);
      \coordinate (a1) at (30:1);
      \coordinate (a2) at (60:1);
      \coordinate (a3) at (90:1);
      \coordinate (a4) at (135:1);
      \coordinate (a5) at (-135:1);
      \coordinate
(a6) at (-90:1);
      \coordinate (a7) at (-45:1);
      \coordinate (r1) at (120:0.4); 
      \coordinate (r2) at (-60:0.4); 
      \coordinate (r3) at (-135:0.7);
      \draw[sline] (a2) to (r1) to (a4);
      \draw[sline] (r1) to[out=-90,in=150] (r2) to (a6);
      \draw[sline] (r2) to (a1);
      \draw[tline] (a3) to (r1) to[out=-150,in=90] (r3) to (a5);
      \draw[tline] (r1) to[out=-30,in=90, looseness=1.5] (r2) to (a7);
      \draw[tline] (r2) to[out=-150,in=0] (r3);
   \end{tikzpicture} 
   && \text{if } m_{\cols, \colt} = 3 \quad (\text{type } A_2) \text{,} \\[0.5cm]
   %\label{eqTrivalentPull4}
   \begin{tikzpicture}[baseline=(current  bounding  box.center), scale=0.75]
      \draw[spot] (0,0) circle (1cm);
      \coordinate (a0) at (35:1); 
      \coordinate (a1) at (55:1);
      \coordinate (a2) at (75:1);
      \coordinate (a3) at (105:1); 
      \coordinate (a4) at (135:1);
      \coordinate (a5) at (-45:1);
      \coordinate (a6) at (-75:1);
      \coordinate (a7) at (-105:1); 
      \coordinate (a8) at (-135:1);
      \coordinate (l1) at (45:0.4); 
      \coordinate (l0) at (0:0);
      \draw[sline] (a0) -- (l1) -- (a1);
      \draw[sline] (l1) -- (l0) -- (a3);
      \draw[sline] (a6) -- (l0) -- (a8);
      \draw[tline] (a5) -- (l0) -- (a7);
      \draw[tline] (a2) -- (l0) -- (a4);
   \end{tikzpicture} 
   \enspace &= \enspace
   \begin{tikzpicture}[baseline=(current  bounding  box.center), scale=0.75]
      \draw[spot] (0,0) circle (1cm);
      \coordinate (a0) at (35:1);
      \coordinate (a1) at (55:1);
      \coordinate (a2) at (75:1);
      \coordinate (a3) at (105:1);
      \coordinate (a4) at (135:1);
      \coordinate (a8) at (-45:1);
      \coordinate (a7) at (-75:1);
      \coordinate (a6) at (-105:1);
      \coordinate (a5) at (-135:1);
      \coordinate (r1) at (105:0.4); 
      \coordinate (r2) at (-75:0.4); 
      \coordinate (r3) at (-135:0.7);
      \draw[sline] (a1) -- (r1) -- (a3);
      \draw[sline] (r1) to[out=-135,in=90] (r3) to (a5);
      \draw[sline] (r3) to[out=0,in=-135 ] (r2) to (a7);
      \draw[sline] (r1) to[out=-75,in=105] (r2);
      \draw[sline] (r2) to (a0);
      \draw[tline] (a2) -- (r1) -- (a4);
      \draw[tline] (r1) to[out=-105, in=135, looseness=2] (r2)  to[out=75, in=-45, looseness=2] (r1);
      \draw[tline] (a6) to (r2) to (a8);
   \end{tikzpicture}
   && \text{if } m_{\cols, \colt} = 4 \quad (\text{type } B_2) \text{.}
\end{alignat*}

The next two-colour relation is called \emph{Jones-Wenzl Relation} and expresses
a $2m_{s,t}$-valent with a dot on one strand as a linear combination of diagrams
in which only dots and trivalent vertices appear. We state it only for 
$m_{\cols, \colt} \in \{2, 3, 4\}$ and refer the reader to \cite{E3} for more detail:
\begin{alignat*}{2}
   %\label{eqJW2}
   \begin{tikzpicture}[baseline=(current  bounding  box.center)]
      \draw[spot] (0,0) circle (0.75);     
      \draw[tline] (-45:0.75) -- (135:0.75);
      \draw[sline] (-135:0.3) -- (45:0.75);
      \draw[sdot] (-135:0.3) circle (\circSize);
   \end{tikzpicture} 
   \enspace &= \enspace
   \begin{tikzpicture}[baseline=(current  bounding  box.center)]
      \draw[spot] (0,0) circle (0.75);
      \draw[tline] (-45:0.75) -- (135:0.75);
      \draw[sline] (45:0.3) -- (45:0.75);
      \draw[sdot] (45:0.3) circle (\circSize);
   \end{tikzpicture} 
   \qquad && \text{if  } m_{\cols, \colt} = 2 \text{,} %\: (\text{type } A_1 \times A_1) 
   \\[0.5cm]
   %\label{eqJW3}
   \begin{tikzpicture}[baseline=(current  bounding  box.center)]
      \draw[spot] (0,0) circle (0.75cm);
      \foreach \r in {30, 150, -90}
         \draw[tline] (0,0) -- (\r:0.75);
      \foreach \r in {90,-30}
         \draw[sline] (0,0) -- (\r:0.75);
      \draw[sline] (0,0) -- (-150:0.45);
      \draw[sdot] (-150:0.45) circle (\circSize); 
   \end{tikzpicture} 
   \enspace &= \enspace
   \begin{tikzpicture}[baseline=(current  bounding  box.center)]
      \draw[spot] (0,0) circle (0.75cm);
      \foreach \r in {30, 150, -90}
         \draw[tline] (0,0) -- (\r:0.75);
      \foreach \r in {90,-30}
{ 
         \draw[sline] (\r:0.375) -- (\r:0.75);
         \draw[sdot] (\r:0.375) circle (\circSize);}
   \end{tikzpicture} \; + \;
   \begin{tikzpicture}[baseline=(current  bounding  box.center)]
      \draw[spot] (0,0) circle (0.75cm);
      \draw[tline] (-90:0.75) to[out=90,in=-30] (150:0.75);
      \draw[sline] (90:0.75) to[out=-90,in=150] (-30:0.75);
      \draw[tline] (30:0.375) -- (30:0.75);
      \draw[tdot] (30:0.375) circle (\circSize);
   \end{tikzpicture}
   \qquad && \text{if } m_{\cols, \colt} = 3 \text{,} %\: (\text{type } A_2) 
   \\[0.5cm]
   %\label{eqJW4}
   \begin{split}
   \begin{tikzpicture}[baseline=(current  bounding  box.center)]
      \draw[spot] (0,0) circle (0.75cm);
      \foreach \r in {0, 90, 180, -90}
         \draw[tline] (0,0) -- (\r:0.75);
      \foreach \r in {45, 135, -45}
         \draw[sline] (0,0) -- (\r:0.75);
      \draw[sline] (0,0) -- (-135:0.45);
      \draw[sdot] (-135:0.45) circle (\circSize); 
   \end{tikzpicture} 
   \enspace &= \enspace
   \begin{tikzpicture}[baseline=(current  bounding  box.center)]
      \def \r {0.45}
      \draw[spot] (0,0) circle (0.75cm);
      \foreach \a in {-90, 45, 135}
         \draw[sline] (\a:0.75) to (\a:\r);
      \foreach \a in {90, -45}
         \draw[tline] (\a:0.75) to (\a:\r);
      \draw[tline] (-120:0.75) to[out=80, in=-45] (157.75:0.75);
      \draw[sline] (-90:\r) to[out=90, in=-45] (135:\r);
      \draw[tline] (90:\r) to[out=-90, in=135] (-45:\r);
      \draw[sdot] (45:\r) circle (\circSize); 
   \end{tikzpicture} \; + \;
   \begin{tikzpicture}[baseline=(current  bounding  box.center)]
      \def \r {0.45};
      \draw[spot] (0,0) circle (0.75cm);
      \foreach \a in {-90, 45, 135}
         \draw[sline] (\a:0.75) to (\a:\r);
      \foreach \a in {90, -45}
         \draw[tline] (\a:0.75) to (\a:\r);
      \draw[tline] (-120:0.75) to (157.75:0.75);
      \draw[tline] (-135:\r) to (-135:0.6);
      \draw[sline] (-90:\r) to[out=90, in=-135] (45:\r);
      \draw[tline] (-135:\r) to[out=45, in=-90] (90:\r);
      \draw[sdot] (135:\r) circle (\circSize); 
      \draw[tdot] (-45:\r) circle (\circSize); 
   \end{tikzpicture} \; + \;
   \begin{tikzpicture}[baseline=(current  bounding  box.center)]
      \def \r {0.45};
      \draw[spot] (0,0) circle (0.75cm);
      \foreach \a in {-90, 45, 135}
         \draw[sline] (\a:0.75) to (\a:\r);
      \foreach \a in {90, -45}
         \draw[tline] (\a:0.75) to (\a:\r);
      \draw[tline] (-120:0.75) to (157.75:0.75);
      \draw[tline] (-135:\r) to (-135:0.6);
      \draw[sline] (135:\r) to[out=-45,in=-135] (45:\r);
      \draw[tline] (-135:\r) to[out=45,in=135] (-45:\r);
      \draw[tdot] (90:\r) circle (\circSize); 
      \draw[sdot] (-90:\r) circle (\circSize); 
   \end{tikzpicture} \\ & - \langle \rt{\cols}, \cort{\colt} \rangle \;
   \begin{tikzpicture}[baseline=(current  bounding  box.center)]
      \def \r {0.45}
      \draw[spot] (0,0) circle (0.75cm);
      \foreach \a in {-90, 45, 135}
         \draw[sline] (\a:0.75) to (\a:\r);
      \foreach\a in {90, -45}
         \draw[tline] (\a:0.75) to (\a:\r);
      \draw[tline] (-120:0.75) to (157.75:0.75);
      \draw[tline] (-135:\r) to (-135:0.6);
      \foreach \a in {-135, -45, 90}
         \draw[tline] (\a:\r) to (0,0);
      \draw[sdot] (45:\r) circle (\circSize);
      \draw[sdot] (-90:\r) circle (\circSize);
      \draw[sdot] (135:\r) circle (\circSize); 
   \end{tikzpicture} \; - \langle \rt{\colt}, \cort{\cols} \rangle \;
   \begin{tikzpicture}[baseline=(current  bounding  box.center)]
      \def \r {0.45};
      \draw[spot] (0,0) circle (0.75cm);
      \foreach \a in {-90, 45, 135}
         \draw[sline] (\a:0.75) to (\a:\r);
      \foreach \a in {90, -45}
      \draw[tline] (\a:0.75) to (\a:\r);
      \draw[tline] (-120:0.75) to[out=80,in=-45] (157.75:0.75);
      \foreach \a in {135, 45, -90}
         \draw[sline] (\a:\r) to (0,0);
      \foreach \a in {-45, 90}
         \draw[tdot] (\a:\r) circle (\circSize); 
    \end{tikzpicture}
    \end{split}
    \qquad && \text{if } m_{\cols, \colt} = 4 \text{.} %\: (\text{type } B_2)
\end{alignat*}

\subsubsection{Three-Colour Relations}
We do not repeat the definition of the Zamolodchikov relations or ``higher 
braid relations'' here. The reader can find them in \cite[\S 1.4.3]{EW2} 
and is referred to \cite{EW3} for more detail on the topological origins of 
the Zamolodchikov relations.

\subsection{Light Leaves and Double Leaves}
\label{secLL}
In this section we briefly discuss how to describe bases for morphism spaces 
in $\BS$. Fix an expression $\expr{w}$ and a reduced expression $\expr{x}$. 
In \cite[\S 6.1]{EW2} it is described how one may associate a ``light leaves 
morphism''  $\LL{w}{e} \in \Hom_{\BS}(\expr{w}, \expr{x})$ to each 
subexpression $\expr{e}$ of $\expr{w}$ such that $(\expr{w}^{\expr{e}})_{\bullet} = x$. 
We will not need the explicit construction here, but the reader is encouraged 
to consult \cite[\S 6.1]{EW2} to follow our examples. The construction of 
light leaves follows a construction of Libedinsky for Soergel bimodules 
\cite{Li1} and depends on certain additional non-canonical choices.

In the special case of $x = e$, the identity of $W$, we get (see \cite[Proposition 6.12]{EW2}):
\begin{prop}
   The set of all light leaves indexed by subsequences $\expr{e}$ of $\expr{w}$ expressing
   the identity of $W$ gives an $R$-basis of $\Hom_{\BS}(\expr{w}, \varnothing)$.
\end{prop}

For an $S$-graph $D$ denote by $\overline{D}$ the $S$-graph obtained by flipping 
the diagram upside down. This induces a contravariant equivalence on the monoidal 
category $\BS$ fixing all objects. 

Out of light leaves one can construct double leaves as follows. Let $\expr{x}$ 
and $\expr{y}$ be arbitrary expressions in $S$. For any subsequences $\expr{e}$ 
(resp. $\expr{f}$) of $\expr{x}$ (resp. $\expr{y}$) both expressing $w \in W$ 
define $\dL{w}{e}{f} \defeq \overline{\LL{y}{f}} \circ \LL{x}{e}$. The following
result can be found in \cite[Theorem 6.11]{EW2} (and was proved earlier
in the setting of Soergel bimodules by Libedinsky in \cite{LiLLLusConj}):

\begin{thm}
   \label{thmdL}
   The set of all double leaves ranging over all $w \in W$ and pairs of
   subsequences $\expr{e}$ (resp. $\expr{f}$) of $\expr{x}$ (resp. $\expr{y}$)
   both expressing $w$ gives an $R$-basis of $\Hom_{\BS}(\expr{x}, \expr{y})$.
\end{thm}

%\subsection{The Diagrammatic Category of Soergel bimodules and its Properties}
\subsection{The Diagrammatic Category: Properties}

Note that all relations in $\BS$ are homogeneous for our grading on $S$-graphs
and thus $\BS$ is a category enriched in graded left $R$-modules; multiplying
an $S$-graph $D$ with a homogeneous polynomial $f \in R$ from the left is defined
by decorating the leftmost region of $D$ with $f$.

Let $\HC$ be the Karoubian envelope of the graded version of the additive closure
of $\BS$, in symbols $\HC = \karalg{\BS}$. We call $\HC$ the 
\emph{diagrammatic category of Soergel bimodules}. In other words, in the passage 
from $\BS$ to $\HC$ we first allow direct sums and grading shifts (restricting 
to degree preserving homomorphisms) and then the taking of direct summands. The 
following properties can be found in \cite[Lemma 6.24, Theorem 6.25 and 
Corollary 6.26]{EW2}:

\begin{thmlab}[Properties of $\HC$]
   \label{thmDiagProps}
   Let $k$ be a complete local ring (e.g. a field or the $p$-adic 
   integers $\Z_p$).
   \begin{enumerate}
      \item $\HC$ is a Krull-Schmidt category.
      \item For all $w \in W$ there exists a unique, indecomposable object $B_w 
            \in \HC$ which is a direct summand of $\expr{w}$ for any reduced 
            expression $\expr{w}$ of $w$ and which is not isomorphic to a grading 
            shift of any direct summand of any expression $\expr{v}$ for $v < w$. 
            The object $B_w$ does not depend up to isomorphism on the reduced 
            expression $\expr{w}$ of $w$.
      \item The set $\{ B_w \; \vert \; w \in W\}$ gives a complete set of 
            representatives of the isomorphism classes of indecomposable objects 
            in $\HC$ up to grading shift.
      \item There exists a unique isomorphism of $\Z[v, v^{-1}]$-algebras
            \[ \ch: \Galg{\HC} \longrightarrow \heck \]
            sending $[B_s]$ to $\kl{s}$ for all $s \in S$, where $\Galg{\HC}$ 
            denotes the split Grothendieck group of $\HC$. (We view $\Galg{\HC}$ as a
            $\Z[v, v^{-1}]$-algebra as follows: the monoidal structure
            on $\HC$ induces a unital, associative multiplication and $v$
            acts via $v[B] \defeq [B(1)]$ for an object $B$ of $\HC$.)
   \end{enumerate}
\end{thmlab}

It should be noted that we do not have a diagrammatic presentation of $\HC$ 
as determining the idempotents in $\BS$ is usually extremely difficult. 

Observe that $\overline{(-)}$ extends to a contravariant equivalence of the
graded, $R$-linear, additive, monoidal category $\HC$ sending $B_w (n)$ to
$B_w (-n)$ for all $n \in \Z$ and $w \in W$.

In order to explicitly give the isomorphism in the last part of \Cref{thmDiagProps}, 
we need to introduce some more notation. For $x \in W$, let $\HCnl{x}$ be the
quotient category of $\HC$ by the $2$-sided ideal of morphisms factoring 
through any grading shift of a reduced expression $\expr{y}$ for some $y < x$. 
Write $\Homnl{x}(-,-)$ for homomorphism spaces in $\HCnl{x}$. In $\HCnl{x}$ any
two reduced expressions for $x$ become canonically isomorphic. 
We denote the image of any reduced expression for $x$ in $\HCnl{x}$ by $x$ 
as well. Under the assumptions of \Cref{thmDiagProps} we can define the 
\emph{diagrammatic character} on an object $B$ of $\HC$ as follows:
\begin{alignat*}{1}
   \ch: \Galg{\HC} &\longrightarrow \heck \text{,}\\
               [B] &\longmapsto \sum_{w \in W} \grk 
               \Homnl{w}^{\bullet}(B, w) \std{w}
\end{alignat*}
and extend $\Z[v, v^{-1}]$-linearly. In the last definition $\grk$ denotes 
the graded rank of the free $R$-module of homomorphisms of all degrees  
\[ \Homnl{w}^{\bullet}(B_x, w) \defeq \bigoplus_{n \in \Z} 
   \Homnl{w}(B_x, w(n)) \text{.} \]

\section{The \texorpdfstring{$p$}{p}-Canonical Basis and Intersection Forms}
\label{secIntForm}

In this section we recall the definition of the $p$-canonical basis of 
the Hecke algebra (see \cite{WParityOberwolfach}) and explain how to 
calculate it using intersection forms. Fix a field $k$ of characteristic 
$p \geqslant 0$ and the realization ${}^{k} V$ of $(W, S)$. We use 
this realization to define $\HC$.

It is an interesting question what basis of the Hecke algebra the classes of
the self-dual indecomposable objects in $\HC$ correspond to.
The answer is given for $k = \R$ by \emph{Soergel's conjecture} which
Elias and the second author proved in \cite{EW1}. 
%Note that the crucial positivity assumption in \cite[(3.1)]{EW1} is 
%automatically satisfied as the simple coroots pairwise form obtuse angles.
\begin{thmlab}[Elias-W. 2013]
   \label{thmSoergelConj}
   $\ch([B_w]) = \kl{w}$ for all $w \in W$.
\end{thmlab}

This illustrates that the basis of self-dual indecomposable objects in $\HC$ 
gives an extremely interesting basis of $\heck$ for $k = \R$ and 
motivates our definition of the $p$-canonical basis.

\begin{defn}
   Define $\pcan{w} = \ch([B_w])$ for all $w \in W$ where $\ch: \Galg{\HC} 
   \overset{\cong}{\longrightarrow} \heck$ is the isomorphism of $\Z[v, v^{-1}]$-algebras
   introduced earlier, and $p$ denotes the characteristic of $k$ as above.
\end{defn}

\Cref{thmDiagProps} implies that $\{ \pcan{w} \; \vert \; w \in W \}$
gives a basis of $\heck$, called the \emph{$p$-canonical basis}. As will become 
clearer later, the $p$-canonical basis depends only on the type of the
root system chosen and on the characteristic $p$ of the field $k$.

Next, we are going to explain how to use intersection forms to explicitly calculate
the $p$-canonical basis.  In order to calculate $\pcan{w}$ we proceed by 
induction on $l(w)$. The induction start is given by $\pcan{e} = \kl{e} = \std{e}$. 
Assume that we have already calculated $\pcan{v}$ for all $v < w$. Let 
$\expr{w}$ be an arbitrary reduced expression for $w$. According to 
\Cref{thmDiagProps}, we need to decompose $\expr{w}$ into indecomposable 
objects $B_x(n)$ for $x \in W$ and $n \in \Z$ in $\HC$. 
For this we need local intersection forms:

Write $\Homnl[k]{x}(-,-)$ for the homomorphism spaces in $k \otimes_R \HCnl{x}$
where we kill the action of the unique maximal ideal of all polynomials of positive
degree. Since $\Homnl{x}^{\bullet}(\expr{w},x)$ is a graded free $R$-module with 
basis (see \Cref{secLL}):
\[ \{ \LL{w}{e} \; \vert \; \expr{e} \text{ is a subexpression of } \expr{w}
 \text{ expressing } x \} \text{,} \]
$\Homnl[k]{x}^{\bullet}(\expr{w},x)$ is a graded $k$-vector space on the 
same basis.

For an arbitrary expression $\expr{w}$ in $S$ and $x \in W$, consider
the $k$-bilinear $\Hom$-pairing
\begin{alignat*}{2}
   \Homnl[k]{x}^{\bullet}(x, \expr{w}) &\times \Homnl[k]{x}^{\bullet}(\expr{w},x) 
      &&\longrightarrow \Endnl[k]{x}^{\bullet}(x) = k \text{,} \\
   (f &, g) &&\longmapsto g \circ f \text{.}
\end{alignat*}
Observe that $\Endnl[k]{x}^{\bullet}(x)$ is concentrated in degree $0$ and that 
the duality $\overline{(-)}$ on $\HC$ gives an isomorphism between 
$\Homnl[k]{x}^{\bullet}(x, \expr{w})$ and 
$\overline{\Homnl[k]{x}^{\bullet}(\expr{w},x)}$. This allows us to define:
 
\begin{defn}
   The \emph{local intersection form of $\expr{w}$ at $x$} is the $k$-bilinear
   pairing on the graded free $k$-vector space 
   $\Homnl[k]{x}^{\bullet}(\expr{w},x)$ given by
   \begin{alignat*}{2}
      I_{\expr{w}, x}: \Homnl[k]{x}^{\bullet}(\expr{w},x) &\times 
                       \Homnl[k]{x}^{\bullet}(\expr{w},x) &&\longrightarrow 
                       \Endnl[k]{x}(x) = k \text{,}\\
                       (f &, g) &&\longmapsto g \circ \overline{f} \text{.}
   \end{alignat*}
\end{defn}

The local intersection form of $\expr{w}$ at $x$ can be split up into degree pieces
as follows: Since $\Endnl[k]{x}^{\bullet}(x)$ is concentrated in degree $0$,
a homomorphism in $\Homnl[k]{x}^{\bullet}(\expr{w},x(d))$ for some $d \in \Z$
can only pair non-trivially with elements of $\Homnl[k]{x}^{\bullet}(\expr{w},x(-d))$.
The \emph{$d$-th grading piece of the intersection form} can thus be defined
as:
\begin{alignat*}{2}
  I_{\expr{w}, x}^d: \Homnl[k]{x}(\expr{w}, x(-d)) &\times \Homnl[k]{x}(\expr{w},x(d))
        &&\longrightarrow \Endnl[k]{x}(x(d)) = k
\end{alignat*}

Finally, the \emph{graded rank} of $I_{\expr{w}, x}$ is denoted by 
$n_{x, w} \in \Z_{\geqslant 0}[v, v^{-1}]$ and defined as
\[ n_{x, w} \defeq \sum_{d \in \Z} \rk(I_{\expr{w}, x}^d) v^d \text{.} \]
The following lemma illustrates the importance of intersection
forms for the calculation of the $p$-canonical basis and follows 
from an argument similar to \cite[Lemma 3.1]{JMW1}: 

\begin{lem}
   The multiplicity of $B_x$ in $\expr{w}$ in $\HC$ is given
   by the graded rank of $I_{\expr{w}, x}$.
\end{lem}

After calculating the graded ranks of all $I_{\expr{w}, x}$ for $x < w$,
we can write for $\expr{w} = s_1 s_2 \dots s_n$:
\[ \kl{s_1} \kl{s_2} \dots \kl{s_n} = \pcan{w} + \sum_{x < w} n_{x, w} \pcan{x} \text{.} \]

\begin{remark}
By comparing the intersection forms over $\Q$ and $k$, one may deduce that 
one only needs to calculate the graded ranks of $I_{\expr{w}, x}$
for those $x$ such that $\kl{x}$ occurs with a non-trivial coefficient
when expressing $\kl{s_1} \kl{s_2} \dots \kl{s_n}$ in terms of the Kazhdan-Lusztig 
basis.
\end{remark}

In order to determine $\pcan{w}$ we have to invert the matrix 
$(n_{x, y})_{x, y \leqslant w}$ which is upper triangular with ones on the 
diagonal in any total order refining the Bruhat order.

Finally, it should be noted that in practice one calculates the intersection 
form once over $\Z$ and reduces modulo different primes.

\begin{remark}
   There are other ways to calculate the $p$-canonical basis (which, however, 
   are much more difficult in practice).
   \begin{enumerate}
      \item If one can describe the geometry of the corresponding Schubert
            varieties quite explicitly, one can do calculations using parity
            sheaves (see \cite[appendix]{BW}).
      \item In \cite{FW} Fiebig and the second author show that for a field $k$
            of characteristic $p$ (or more generally a complete local PID), the 
            Braden-MacPherson algorithm on the Bruhat graph allows one to compute 
            the $p$-canonical basis.
   \end{enumerate}
\end{remark}

\subsection{Calculations in the nil Hecke Ring}
\label{secNilHeck}

In \cite{HW} Xuhua He and the second author reduce the calculation of certain
entries in the intersection form to a simple formula in the nil Hecke ring. Instead
of going into too much detail, we will try to give a survey of these results.

First, recall the definition of the nil Hecke ring. Let $Q$ be the field of fractions
of $R$. Denote by $Q * W$ the smash product. In other words, $Q * W$ is
the free left $Q$-module with basis $\{\delta_w \; \vert \; w \in W \}$ and
multiplication given by
\[ (f \delta_x) (g \delta_y) = f(xg) \delta_{xy} \]
for $f, g \in Q$ and $x, y \in W$. Inside $Q * W$, we consider the elements
\[ D_s = \frac{1}{\rt{s}} (\delta_{e} - \delta_s) 
       = (\delta_{e} + \delta_{s}) \frac{1}{\rt{s}} \]
which satisfy the following relations:
\begin{align*}
   D_s^2 &= 0 \text{,}\\
   D_s f &= (sf) D_s + \partial_s(f) \quad \text{for all } f \in Q \text{,}\\
   \underbrace{D_s D_t D_s \dots}_{m_{s, t} \text{ terms}} &= 
      \underbrace{D_t D_s D_t \dots}_{m_{s,t} \text{ terms}} \ \text{.}
\end{align*}
The last relation ensures that for $y \in W$ and any reduced expression
$\expr{y} = (s_1, s_2, \dots, s_n)$ of $y$ we get well-defined elements
\[ D_y = D_{s_1} D_{s_2} \dots D_{s_n} \in Q * W  \text{.}\]
The nil Hecke ring $\NH$ is the left $R$-submodule of $Q * W$ generated by 
$\{ D_y \; \vert \; v \in W\}$.
% It can be shown that $\NH$ is a free left $R$-module with basis
% $\{ D_y \; \vert \; v \in W\}$.

Next, we briefly introduce gobbling morphisms. For any expression 
$\expr{w} = (s_1, s_2, \dots, s_n)$ in $S$, consider the following
$01$-sequence $\expr{e}$ with:
\[ e_i = 
   \begin{cases}
      1 & \text{if } w_{i-1} s_i > w_{i-1} \text{,} \\
      0 & \text{otherwise,}
   \end{cases}
\] where at each step $w_i$ is defined as in \Cref{secCox}.
Note that $(\expr{w}^{\expr{e}})_{\bullet}$ is the maximal element in $W$ expressible
as a subexpression of $\expr{w}$, and that the decoration of $\expr{e}$ consists 
entirely of $U1$'s and $D0$'s. Therefore any choice of light leaf morphism
$\LL{w}{e}$ has degree $l((\expr{w}^{\expr{e}})_{\bullet}) - l(\expr{w})$ and 
consists only of $2m_{st}$-valent and trivalent vertices. Denote by $G_{\expr{w}}$ the
image of $\LL{w}{e}$ in $\HCnl{(\expr{w}^{\expr{e}})_{\bullet}}$. The morphism 
$G_{\expr{w}}$ is called a \emph{gobbling morphism} and can be characterized 
as follows (see \cite[Proposition 3.4]{HW}):

\begin{prop}
   Let $\expr{w}$, $\expr{e}$ be as above. Any morphism 
   $\expr{w} \rightarrow (\expr{w}^{\expr{e}})_{\bullet}$ in 
   $\HCnl{(\expr{w}^{\expr{e}})_{\bullet}}$ given by diagrams consisting only 
   of $2m_{st}$-valent vertices and $l(\expr{w}) - 
   l((\expr{w}^{\expr{e}})_{\bullet})$ trivalent vertices is equal to $G_{\expr{w}}$.
\end{prop}

From this they deduce the canonicity of any light leaf morphism $\LL{w}{f}$
in $\HCnl{(\expr{w}^{\expr{f}})_{\bullet}}$ indexed by a $01$-sequence $\expr{f}$ 
without $D1$'s in its decoration. (This follows because the morphism is given 
as the composition of a sequence of dots on strands corresponding to $U0$'s 
followed by a gobbling morphism.)

Finally, we come to their formula in the nil Hecke ring for certain entries of 
the intersection form. Let $\expr{e^1}$ and $\expr{e^2}$ be two subexpressions of 
$\expr{w}$.  Assume that $\expr{e^1}$ and $\expr{e^2}$ both express the same element
$x \in W$ (i.e. $x = (\expr{w}^{\expr{e^1}})_{\bullet} = (\expr{w}^{\expr{e^2}})_{\bullet}$)
and that their decorations do not contain any $D1$. Define an element of the nil Hecke 
ring as the product $f(\expr{e^1}, \expr{e^2}) = f_1 f_2 \dots f_m$ where
\[ f_i = 
   \begin{cases}
      \rt{s_i} & \text{if } e^1_i = e^2_i = U0 \text{,} \\
      1 & \text{if exactly one of } e^1_i \text{ and } e^2_i \text{ is } U0 \text{,} \\
      D_{s_i} & \text{otherwise.}
   \end{cases}
\]
Denote by $d(\expr{e^1}, \expr{e^2}) \in R$ the coefficient of 
$D_{(\expr{w}^{\expr{e^1}})_{\bullet}}$ in $f(\expr{e^1}, \expr{e^2})$. The main
result is \cite[Theorem 5.1]{HW}:

\begin{thm}
   For $\expr{e^1}$ and $\expr{e^2}$ as above, we have
   \[ I_{\expr{w}, x}(\LL{w}{e^1}, \LL{w}{e^2}) = d(\expr{e^1}, \expr{e^2}) \text{.}\]
\end{thm}

This theorem gives a combinatorial formula for some entries in the intersection 
form. Sometimes one is lucky, and it can be used to calculate the complete 
intersection form, as we will see in examples below.

\section{First Properties of the \texorpdfstring{$p$}{p}-Canonical Basis}
\label{secProp}

The goal of this section is to prove elementary properties of the $p$-canonical
basis and to compare it to the Kazhdan-Lusztig basis. For this we need a 
$p$-modular system. Let $\rO$ be a complete local ring with residue field $k$ of 
characteristic $p > 0$ and quotient field $\K$ of characteristic $0$. Fix the
realization $\pre{\rO}{V}$ of $(W, S)$ and use it to define $\HC$. For $x \in W$ 
we will denote by $B_x$ (resp. $\pre{k}{B}_x$ or $\pre{\K}{B}_x$) the 
indecomposable object in $\HC$ (resp. $\pre{k}{\HC} \defeq \HC \otimes_{\rO} k$ 
or $\pre{\K}{\HC} \defeq \HC \otimes_{\rO} \K$).

The following lemma shows that indecomposable objects in $\HC$ remain
indecomposable when passing to $\pre{k}{\HC}$.
\begin{lem}
   \label{lemModRed}
   We have for all $w \in W$: $B_w \otimes_{\rO} k \cong \pre{k}{B}_w$.
\end{lem}
\begin{proof}
   Assume $B_w \otimes_{\rO} k$ is not indecomposable in $\pre{k}{\HC}$. Then
   there exists a non-trivial idempotent $e \in \End_{\pre{k}{\HC}}(B_w \otimes_{\rO} k)$.
   Since $\End_{\HC}(B_w)$ is a finitely generated $\rO$-module, we can use 
   idempotent lifting techniques for complete local rings (see 
   \cite[Theorem 21.31]{Lam}) and find an idempotent $\tilde{e} \in 
   \End_{\HC}(B_w)$ mapping to $e$ in 
   $\End_{\pre{k}{\HC}}(B_w \otimes_{\rO} k) \cong  \End_{\HC}(B_w) \otimes_{\rO} k$. 
   Since $B_w$ is indecomposable, this idempotent has to be trivial, a 
   contradiction.
\end{proof}

Some of the following properties can also be found in \cite{BW} and \cite{WRedCharVarTypeA}:

\begin{prop}
   \label{propPCanProps}
   For all $x, y \in W$ we have:
   \begin{enumerate}
      \item $\widebar{\pcan{x}} = \pcan{x}$, i.e. $\pcan{x}$ is self-dual,
      \item $\pcan{x} = \std{x} + \sum_{y < x} \p{h}_{y, x} \std{y}$ 
            with $\p{h}_{y, x} \in \Z_{\geqslant 0}[v, v^{-1}]$,
      \item $\pcan{x} = \kl{x} + \sum_{y < x} \p{m}_{y, x} \kl{y}$
            with self-dual $\p{m}_{y, x} \in \Z_{\geqslant 0}[v, v^{-1}]$,
      \item $\p{m}_{y, x} = \p{m}_{y^{-1}, x^{-1}}$ for all $x, y \in W$,
      \item $\p{m}_{y, x} = 0$ unless $\desc{L}(x) \subseteq \desc{L}(y)$
            and $\desc{R}(x) \subseteq \desc{R}(y)$ where $\desc{L}$ and
            $\desc{R}$ denote the left and right descent sets,
      \item $\pcan{x} \pcan{y} = \sum_{z \in W} \p{\mu}^{z}_{x, y} \pcan{z}$
            with self-dual $\p{\mu}^{z}_{x, y} \in \Z_{\geqslant 0}[v, v^{-1}]$,
      \item $\pcan{x} = \kl{x}$ for  $p \gg 0$ (i.e. there are only finitely many primes
            for which $\pcan{x} \neq \kl{x}$).
   \end{enumerate}
\end{prop}
\begin{proof}
   \begin{enumerate}
   \item We proceed by induction on $l(x)$. For small $x$ the statement is clear as
         $\pcan{e} = \kl{e}$ and $\pcan{s} = \kl{s}$ for all $s \in S$ and all 
         primes $p$. Assume that we have shown that $\pcan{y}$ is self-dual for 
         $y < x$. Choose $s \in \desc{L}(x)$ and set $y = sx$. The characterization
         of $\pre{k}{B}_x$ in \Cref{thmDiagProps} implies that $\pre{k}{B}_x$ occurs
         with multiplicity one in $\pre{k}{B}_s \pre{k}{B}_y$
         and that $\overline{\pre{k}{B}_x(n)} = \pre{k}{B}_x(-n)$. Thus we can write
         \[ \pre{k}{B}_s  \pre{k}{B}_y = \pre{k}{B}_x \oplus 
            \bigoplus_{\substack{z < x \\ n \in \Z}} (\pre{k}{B}_z(n))^{\oplus a_{z, n}} \]
         with $a_{y, n} \in \Z_{\geqslant 0}$ for all $y < x$ and $n \in \Z$ and
         all but finitely many of the $a_{y, n}$ are zero. Applying the duality 
         $\overline{(-)}$ to both sides and using that the left hand side is self-dual 
         yields $a_{z, n} = a_{z, -n}$ for all $z < x$ and $n \in \Z$. 
         This implies
         \[ \pcan{s} \pcan{y} = \pcan{x} + \sum_{z < x} \p{\mu}_{z, y} \pcan{z} \]
         where $\p{\mu}_{z, y} = \sum_{n\in Z} a_{z, n}v^n \in \Z_{\geqslant 0}[v, v^{-1}]$ 
         is self-dual for all $z < x$. Therefore the self-duality of the
         left-hand side and our induction hypothesis imply the self-duality
         of $\pcan{x}$.
   \item The unicity in the characterization of $\pre{k}{B}_x$ in \Cref{thmDiagProps} 
         implies that it occurs with multiplicity $1$ in any reduced expression
         $\expr{x}$ of $x$. In the quotient $\HCnl{x}$ all other direct summands
         of $x$ are killed. Thus we get $\grk \Homnl{x}(B_x, x) = 1$. Note that
         the Laurent polynomials $\p{h}_{y, x}$ have non-negative 
         coefficients as they are given by graded ranks of free $R$-modules. 
   \item According to (2), $\kl{x}$ occurs precisely with coefficient $1$ in 
         $\pcan{x}$. The self-duality of the Laurent polynomials $\p{m}_{y, x}$
         follows from (1) and the self-duality of the Kazhdan-Lusztig basis.
         Since $\Homnl{w}^{\bullet}(B_x, w)$ is a free $R$-module, we get for
         $F \in \{k, \K\}$:
         \[ \Homnl{w}^{\bullet}(B_x, w) \otimes_{\rO} F \cong 
            \Hom_{\HCnl{w} \otimes_{\rO} F}^{\bullet}(B_x \otimes_{\rO} F, 
            w \otimes_{\rO} F) \text{.} \]
         This implies in particular using \Cref{lemModRed}:
         \[ \ch(\pre{k}{B}_x) = \ch(B_x) = \ch(B_x \otimes_{\rO} \K) \text{.} \]
         Thus the $\p{m}_{y, x}$ have non-negative coefficients as they come
         from decomposing $B_x \otimes_{\rO} \K$ into indecomposable objects of the
         form $\pre{\K}{B}_x$ in $\pre{\K}{\HC}$ whose character is given by $\kl{x}$
         by \cref{thmSoergelConj}.
   \item There is an equivalence $\iota$ on $\pre{k}{\HC}$ viewed as a $k$-linear category
         induced by the horizontal flip of Soergel graphs. It maps 
         $(s_1, s_2, \dots, s_m) \in \pre{k}{\HC}$ to $(s_m, \dots, s_2, s_1)$ and thus 
         $\pre{k}{B}_x$ to $\pre{k}{B}_{x^{-1}}$ for all $x \in W$.
         It is easy to see that $\iota$ descends to a well-known 
         $\Z[v, v^{-1}]$-linear anti-involution on $\heck$ sending $\kl{x}$ to 
         $\kl{x^{-1}}$ as well as $\pcan{x}$ to $\pcan{x^{-1}}$ for all 
         $x \in W$. By slight abuse of notation we will denote this map 
         also by $\iota$. Thus we have: $\ch \circ \iota = \iota \circ \ch$.
         Finally expressing $\pcan{x}$ in the Kazhdan-Lusztig basis and applying
         $\iota$ proves the claim after comparing coefficients.         
   \item The statement for left descent sets follows from \Cref{lemMultFormPCan} below
         and the fact that the Laurent polynomials $\p{m}_{y, x}$ have non-negative 
         coefficients. Using (4) we can reduce the statement about the right
         descent sets to the case we have just proven.
   \item This follows immediately from the analogue of Soergel's categorification
         theorem (part (4) of \Cref{thmDiagProps}). Indeed, in order to express 
         $\pcan{x} \pcan{y}$ in the $p$-canonical basis, we need to decompose
         $\pre{k}{B}_x \pre{k}{B}_y$ into indecomposable
         objects in $\pre{k}{\HC}$ and thus the Laurent polynomial 
         $\p{\mu}^{z}_{x, y}$ encodes the graded multiplicity of
         $\pre{k}{B}_z$ in this tensor product. Therefore $\p{\mu}^{z}_{x, y}$
         has non-negative coefficients. The self-duality of these Laurent polynomials
         follows from (1).
   \item As explained in \Cref{secIntForm} we need to calculate the graded rank
         of finitely many local intersection forms in order to calculate $\pcan{x}$.
         The rank of each of these intersection forms can only decrease for finitely
         many primes. (7) now follows.
   \end{enumerate}
\end{proof}

The \emph{multiplication formula} from \cite[(2.3.a) and (2.3.c)]{KL}) reads 
for $x \in W$ and $s \in S$ as follows:
\begin{equation}
   \label{eqnMultForm}
   \kl{s} \kl{x} =
   \begin{cases}
      (v + v^{-1})\kl{x} & \text{if } sx < x \text{,}\\
      \kl{sx} + \sum_{\substack{y < x \text { s.t.}\\ sy < y}}{\mu(y, x)\kl{y}} 
         & \text{otherwise.}
   \end{cases}
\end{equation}
where $\mu(y, x)$ is the coefficient of $v$ in the Kazhdan-Lusztig polynomial
$\pre{0}{h}_{y, x}$. One remnant of this multiplication formula for the 
$p$-canonical basis is the following result:
\begin{lem}
   \label{lemMultFormPCan}
   For $x \in W$ and $s \in \desc{L}(x)$ we have:
   \[ \pcan{s} \pcan{x} = (v + v^{-1}) \pcan{x} \text{.} \]
\end{lem}
\begin{proof}
%    Alternative Proof for the first part:
%
%    First observe that $\p{m}_{sx, x} = 0$. Indeed, choose a reduced expression 
%    $\expr{x}$ starting in $s$ and decompose $\expr{x}$ in $\pre{\K}{\HC}$.
%    Note that there is a unique subexpression of $\expr{x}$ giving $sx$, namely
%    $(0, 1, \dots, 1)$ of defect one. Thus $\pre{\K}{B}_{sx}$ does not occur in $\expr{x}$.
%    To express $\pcan{x}$ in the Kazhdan-Lusztig basis, we proceed by induction
%    on $l(x)$. Assume that we have determined all the $\p{m}_{z, y}$ for 
%    all $y < x$ and $z \leqslant y$. We obtain an expression of $\kl{\expr{x}}$ 
%    in terms of the KL-basis from the decomposition of $\expr{x}$ in 
%    $\pre{\K}{\HC}$ and subtract from this 
%    \[ \sum_{y < x} n_{y,x} \pcan{y} = \sum_{y < x} n_{y, x} (\kl{y} + 
%       \sum_{z < y} \p{m}_{z, y} \kl{z})\]
%    to determine the $\p{m}_{y, x}$ for $y \leqslant x$. Since $\kl{sx}$ does
%    not occur in the expression of $\kl{\expr{x}}$ in the KL-basis, it does not occur
%    in $\pcan{x}$.

   Since $sx < x$, we can write using \Cref{propPCanProps} (2) for 
   $\ch(\pre{k}{B}_{sx})$:
   \[ \ch(\pre{k}{B}_s \pre{k}{B}_{sx}) = \kl{s} \ch(\pre{k}{B}_{sx}) 
      = H_{x} + vH_{sx} + \sum_{\substack{y < x \\ y \ne sx}} h_y H_y
      = \kl{x} + \sum_{\substack{y < x \\ y \ne sx}} m_y \kl{y}.
   \]
   As $\pre{k}{B}_x$ is a summand of $\pre{k}{B}_s \pre{k}{B}_{sx}$ we can 
   apply \Cref{propPCanProps} (3) to deduce
   \[ \ch(\pre{k}{B}_x) = \kl{x} + \sum_{\substack{y < x \\ y \ne sx}} 
      \pre{p}{m}_{y,x} \kl{y} \text{.} \]
   This implies $\pre{p}{m}_{sx, x} = 0$. Next, we calculate as follows:
   \begin{alignat*}{1}
      \pcan{s} \pcan{x} &= \kl{s} \pcan{x} \\
                        &= \kl{s} (\kl{x} + \sum_{\substack{y < x \\ y \ne sx}} 
                           \p{m}_{y, x} \kl{y}) \\
                        &\in (v+ v^{-1}) \pcan{x} + \sum_{y < x} 
                             \Z_{\geqslant 0}[v, v^{-1}] \pcan{y}
   \end{alignat*}
   where in the last equality we used \Cref{eqnMultForm} and the observation
   $\p{m}_{sx, x} = 0$ to determine the coefficient in front of $\pcan{x}$.
   After evaluating at $v = 1$ and acting on the trivial module we see
   that no other terms besides $(v+ v^{-1}) \pcan{x}$ can occur on the right hand side.
\end{proof}

\subsection{The Geometric Satake and the \texorpdfstring{$p$}{p}-canonical Basis}
\label{secGeomSatake}

In this section, we will explain the consequences of results about parity 
sheaves on the affine Grassmannian for the $p$-canonical basis. 

A very interesting question is under which hypothesis the indecomposable 
parity sheaves on the affine Grassmannian are perverse. This is 
equivalent to the $p$-canonical basis being a $\Z$-linear combination of 
Kazhdan-Lusztig basis elements.

To shorten our notation, write $\mathcal{K} \defeq \C((t))$ and $\mathcal{O} \defeq
\C[[t]]$. Define the affine Grassmannian $\affGr$ to be the $\Z$-functor given by 
$R \mapsto LG(R) / L^+G(R)$. Its complex points coincide with 
$G(\mathcal{K})/G(\mathcal{O})$. For $\lambda \in \cocharlat$ denote by 
$t^{\lambda} \in G(\mathcal{K})$ the following composition
\[ \Spec(\mathcal{K}) \longrightarrow \Spec(\mathcal{O}) = \mathbb{G}_m  
   \overset{\lambda}{\longrightarrow} T \hookrightarrow G \]
where the first morphism comes from the inclusion $\mathcal{O} \hookrightarrow 
\mathcal{K}$. The Cartan decomposition 
\[ \affGr(\C) = \bigcup_{\lambda \in \cocharlat_+} \underbrace{G(\mathcal{O}) t^{\lambda} 
   G(\mathcal{O}) / G(\mathcal{O})}_{\affGr^{\lambda} \defeq} \]
gives a stratification of $\affGr(\C)$ where each stratum 
$\affGr^{\lambda}$ is a vector bundle over a partial flag variety 
$G(\C) t^{\lambda} G(\mathcal{O}) / G(\mathcal{O})$. Since the 
$G(\mathcal{O})$-orbits are all simply-connected, the indecomposable
parity sheaves are all parametrized by $\lambda \in \cocharlat_+$ (see
\cite[Theorem 4.6]{JMW1}):
\begin{thm}
   Assume that the characteristic of $k$ is not a torsion prime\footnote{See 
   \cite[\S 2.6]{JMW1} for the definition of torsion primes. This 
   restriction can be removed by working in the non-equivariant setting.} for $G$.
   For each $\lambda \in \cocharlat_+$ there exists up to isomorphism a unique
   indecomposable parity complex $\mathcal{E}(\lambda)$ such that 
   $\supp(\mathcal{E}(\lambda)) = \overline{\affGr^{\lambda}}$ and 
   $\mathcal{E}(\lambda) \vert_{\affGr^{\lambda}} = \underline{k}_{\affGr^{\lambda}}
   [\dim \affGr^{\lambda}]$. Every indecomposable parity complex is isomorphic
   to $\mathcal{E}(\lambda)$ for some $\lambda \in \cocharlat_+$.
\end{thm}
Denote by $P_{G(\mathcal{O})}(\affGr(\C), k)$ the $G(\mathcal{O})$-equivariant
perverse sheaves on $\affGr(\C)$ with coefficients in $k$. It comes equipped with
a monoidal structure induced by the convolution product $\ast$.

Recall that we denote by $\Psi = (\charlat, \rts, \cocharlat, \corts)$ the 
root datum of $G$. Define $G^{\vee}$ to be the split connected reductive group scheme 
over $k$ with root datum (Langlands) dual to that of $G$. To each dominant coweight 
$\lambda \in \cocharlat_+$ we can associate an induced representation 
$\nabla(\lambda) \defeq \ind_{B^{\vee}}^{G^{\vee}} k_{\lambda}$, its dual $\Delta(\lambda)$ called 
the \emph{Weyl module} and a simple module $L(\lambda)$ sitting in the 
following sequence $\Delta(\lambda) \twoheadrightarrow L(\lambda) 
\hookrightarrow \nabla(\lambda)$ given by projection to the head and inclusion 
of the socle. Moreover, the set $\{ L(\lambda) \; \vert \; \lambda \in \cocharlat_+ \}$
gives a complete set of representatives for the isomorphism classes of simple, 
rational $G^{\vee}$-modules. The category $\Rep(G^{\vee})$ of rational 
representations of $G^{\vee}$ forms a highest-weight category with Weyl 
modules as standard objects and induced modules as costandard objects. A 
rational representation of $G^{\vee}$ is called \emph{tilting} if it admits two 
filtrations, one with successive quotients isomorphic to Weyl modules and the 
other one with successive quotients isomorphic to induced modules. In 
\cite[Theorem 1.1]{DoTiltingMods} Donkin classifies indecomposable tilting 
modules for algebraic groups giving in our setting:

\begin{thm}
   For each $\lambda \in \cocharlat_+$ there exists up to isomorphism a unique 
   indecomposable tilting module $T(\lambda)$ of highest weight $\lambda$. Moreover,
   $\lambda$ occurs with multiplicity one as a weight of $T(\lambda)$. Every indecomposable
   tilting module is isomorphic to $T(\lambda)$ for some $\lambda \in \cocharlat_+$.
\end{thm}

The Geometric Satake equivalence (see \cite{MV}) gives a monoidal equivalence
\[ (P_{G(\mathcal{O})}(\affGr(\C), k), \ast) \overset{\cong}{\longrightarrow}
   (\Rep(G^{\vee}), \otimes) \text{.} \]
In \cite{JMW3}  Juteau, Mautner and the second author show that if the characteristic
$p$ of $k$ is larger than an explicit bound depending on the root system $\Phi$ of $G$,
then the indecomposable $k$-parity sheaves on the affine Grassmannian are 
perverse. More precisely, they show that for $\lambda 
\in \cocharlat_+$ the indecomposable tilting module $T(\lambda)$ is mapped to 
$\mathcal{E}(\lambda)$ under the geometric Satake equivalence for $p > b(\Phi)$ 
(see below).

From now on let us assume that $G$ is adjoint, so $G^{\vee}$ is simply connected.
In particular, the cocharacter lattice $\cocharlat$ coincides with the coweight
lattice. Denote by $W_a \defeq W \ltimes \Z \corts$ the \emph{affine Weyl group}
and by $W_{ext} \defeq W \ltimes \cocharlat$ the \emph{extended affine Weyl group}.
View  $W_a$ as a Coxeter group $(W_a, S_a)$. We can write
\[ W_{ext} = \Omega \ltimes W_a \]
where $\Omega$ is a finite subgroup of ``length zero elements'' which acts by
automorphisms of the Coxeter system $(W_a, S_a)$ (see \cite[\S 2]{LuQAna}). The
action of $W_{ext}$ on $\cocharlat$ gives bijections:
\begin{alignat*}{3}
   &\cocharlat \  &&\overset{\sim}{\longleftrightarrow} \ &&W_{ext}/W \text{,}\\
    w&(0) &&\longleftrightarrow && w \text{,}\\
   &\cocharlat_+ \  &&\overset{\sim}{\longleftrightarrow} \  W \backslash &&W_{ext} /W
      = \bigcup_{\sigma \in \Omega} \sigma(W) \backslash W_a / W \text{.}
\end{alignat*}
Given $\lambda \in \cocharlat_+$ we denote by $w_{\lambda} \in W_a$ the maximal 
element in the double coset in $\bigcup_{\sigma \in \Omega} \sigma(W) 
\backslash W_a / W$ that corresponds to $\lambda$ under the bijection above.
In this section we will work with the Hecke algebra $\heck[a]$ associated to 
$(W_a, S_a)$.

The next results follows because the projection from the affine flag variety
to the affine Grassmannian is a stratified fiber bundle with fibers isomorphic 
to the finite flag variety $\Fl$: \\

\begin{thm}
   Suppose that all parity sheaves on the affine Grassmannian are perverse.
   Then $\pre{p}{m}_{w_{\lambda}, w_{\nu}} \in \Z$ for all $\lambda, \nu \in 
   \cocharlat_+$.
\end{thm}

The explicit bound has later been improved by Mautner and Riche in \cite{MR}.
They show that the parity sheaves on the affine Grassmannian are perverse 
whenever the characteristic $p$ of $k$ 
is good for $G$, thus giving the following bounds $p > b(\Phi)$ if the root system
$\Phi$ of $G$ is irreducible:\\
\begin{center}
   \begin{tabular}{c || c | c | c | c }
      Type of $\Phi$ & $A_n$ & $B_n$, $C_n$, $D_n$ & $E_6$, $E_7$, $F_4$, $G_2$ & $E_8$ \\
      \hline
      $b(\Phi)$ & $1$ & $2$ & $3$ & $5$ 
   \end{tabular}
\end{center}

The following result gives an interpretation of the multiplicities in the
$p$-canonical basis in terms of $\Rep(G^{\vee})$ (see 
\cite[Corollary 4.1]{JMW3} for the first part):

\begin{lem}
   \label{lemPCanTilt}
   Suppose that the characteristic $p$ of $k$ satisfies $p > b(\Phi)$. Then we 
   have for $\lambda, \mu \in \cocharlat_+$:
   \begin{enumerate}
      \item $\pre{p}{h}_{w_{\mu}, w_{\lambda}}(1) = \dim T(\lambda)_{\mu}$,
      \item $\pre{p}{m}_{w_{\mu}, w_{\lambda}} = 
            [T(\lambda) : \Delta(\mu)] = [T(\lambda) : \nabla(\mu)]$ where 
            $[T(\lambda) : \Delta(\mu)]$ denotes the multiplicity of $\Delta(\mu)$
            in a $\Delta$-flag on $T(\lambda)$.
   \end{enumerate}
\end{lem}

The first part of the last result shows that $\pre{p}{h}_{w_{\mu}, w_{\lambda}}$
gives a $q$-analogue of the weight multiplicity of $\mu$ in $T(\lambda)$. In
characteristic $0$ (where $T(\lambda) = L(\lambda)$) this result can be found in
\cite[Theorem 6.1]{LuQAna}. In particular, knowledge of the characters of
indecomposable tilting modules for $G^{\vee}$ is equivalent to the knowledge
of the $p$-canonical basis elements $\{ \pcan{w_\lambda} \; \vert \; \lambda \in
\cocharlat_+\}$.

\section{Examples}
\label{secEx}

According to the classification of root systems and connected semi-simple algebraic groups,
a Dynkin diagram fixes a semi-simple, adjoint algebraic group $G$
together with a maximal torus $T \subseteq G$ such that the root system determined
by $(G, T)$ corresponds to the given Dynkin diagram. In this section we will only
give the Dynkin diagram and consider the corresponding root datum of this pair
$(G, T)$ together with an arbitrary basis labelled by the nodes of the Dynkin diagram
as input.

\subsection{Type \texorpdfstring{$B_2$}{B2}}
We label the simple reflections as follows:
\[\begin{tikzpicture}[auto, baseline=(current  bounding  box.center)]
      \draw (0,\edgeShift) -- (-1,\edgeShift);
      \draw (0,-\edgeShift) -- (-1,-\edgeShift);
      \path (0,0) to node[Greater] (mid) {} (-1,0);
      \draw (mid.center) to +(30:\wingLen);
      \draw (mid.center) to +(330:\wingLen);
      \node [DynNode] (1) at (-1,0) {$s$};
      \node [DynNode] (2) at (0,0) {$t$}; 
\end{tikzpicture} \]
That means for the pairing between the simple coroots and roots:
\begin{align*}
   \langle \rt{t}, \cort{s} \rangle &= -2 \text{,}\\
   \langle \rt{s}, \cort{t} \rangle &= -1 \text{.}
\end{align*}

Because the Schubert varieties associated to $e, s, t, st, ts$ and $stst$ are 
smooth, we have $\pcan{x} = \kl{x}$ for $x \in \{e, s, t, st, ts, stst\}$
and all primes $p$. (This can also be checked directly.) The remaining two 
elements are $sts$ and $tst$. The two subsequences of $(s, t, s)$ expressing $s$ and 
corresponding light leaves are:
\begin{align*}
   (U1, U0, D0) \  &\text{of defect } 0 
      \ \rightsquigarrow \  L_1 = 
      \begin{tikzpicture}[baseline=(current  bounding  box.center)]
         \draw[sline] (-\xdist,0) -- (-\xdist,\ydist) to[out=90, in=225] (0,2*\ydist) -- (0,3*\ydist);
         \draw[sline] (\xdist,0) -- (\xdist,\ydist) to[out=90, in=315] (0,2*\ydist);
         \draw[tline] (0,0) -- (0,\ydist);
         \draw[tdot] (0,\ydist) circle (\circSize);
      \end{tikzpicture} \text{,} \\
   (U0, U0, U1) \  &\text{of defect } 2 
      \ \rightsquigarrow L_2 = 
      \begin{tikzpicture}[baseline=(current  bounding  box.center)]
         \draw[sline] (-\xdist,0) -- (-\xdist,\ydist);
         \draw[sdot] (-\xdist,\ydist) circle (\circSize);
         \draw[sline] (\xdist,0) -- (\xdist,\ydist) to[out=90, in=270] (0,2*\ydist) -- (0,3*\ydist);
         \draw[tline] (0,0) -- (0,\ydist);
         \draw[tdot] (0,\ydist) circle (\circSize);
      \end{tikzpicture} \text{.}
\end{align*}

Thus the local intersection form of $(s,t,s)$ at $s$ is given by
\[ I_{sts, s} = (L_i \circ \overline{L_j})_{i, j \in \{1, 2\}} = \left (
   \begin{matrix}
      \langle \rt{t}, \cort{s} \rangle & \rt{t} \\
      \rt{t} & \rt{s} \rt{t} \\
   \end{matrix} \right ) \]
where the top left entry comes from the following calculation:
\[ \begin{tikzpicture}[baseline=(current  bounding  box.center)]
      \draw[sline] (0, -2*\ydist) to (0, -\ydist) to[out=135, in=270] (-\xdist,0) 
                   to (-\xdist,\ydist) to[out=90, in=225] (0,2*\ydist) to (0,3*\ydist);
      \draw[sline] (0, -\ydist) to[out=45, in=270] (\xdist, 0) 
                   to (\xdist,\ydist) to[out=90, in=315] (0,2*\ydist);
      \draw[tline] (0,0) to (0,\ydist);
      \draw[tdot] (0,\ydist) circle (\circSize);
      \draw[tdot] (0,0) circle (\circSize);
   \end{tikzpicture} \quad = \quad
   \begin{tikzpicture}[baseline=(current  bounding  box.center), 
                       inner sep=0mm, outer sep=0mm]
      \draw[sline] (0, -2*\ydist) to (0, -\ydist) to[out=135, in=270] (-\xdist,0) 
                   to (-\xdist,\ydist) to[out=90, in=225] (0,2*\ydist) to (0,3*\ydist);
      \draw[sline] (0, -\ydist) to[out=45, in=270] (\xdist, 0) 
                   to (\xdist,\ydist) to[out=90, in=315] (0,2*\ydist);
      \node (v) at (0,\ydist / 2)
{$\rt{\colt}$};
   \end{tikzpicture} \quad = \quad
   \begin{tikzpicture}[baseline=(current  bounding  box.center), 
                       inner sep=0mm, outer sep=0mm]
      \draw[sline] (0, -2*\ydist) to (0, -\ydist) to +(135:\armLen);
      \draw[sdot] ($ (0, -\ydist) + (135:\armLen) $) circle (\circSize);
      \draw[sline] (0,3*\ydist) to (0,2*\ydist) to +(225:\armLen);
      \draw[sdot] ($ (0,2*\ydist) +(225:\armLen) $) circle (\circSize);
      \draw[sline] (0, -\ydist) to[out=45, in=270] (\xdist, 0) 
                   to (\xdist,\ydist) to[out=90, in=315] (0,2*\ydist);
      \node (v) at (-\xdist,\ydist / 2)
{$\partial_{\cols}\rt{\colt}$};
   \end{tikzpicture} \enspace + \enspace
   \begin{tikzpicture}[baseline=(current  bounding  box.center), 
                       inner sep=0mm, outer sep=0mm]
      \draw[sline] (0, -2*\ydist) to (0, -\ydist) to[out=135, in=270] (-\xdist,0) 
                   to (-\xdist,\ydist) to[out=90, in=225] (0,2*\ydist) to (0,3*\ydist);
      \draw[sline] (0, -\ydist) to[out=45, in=270] (\xdist, 0) 
                   to (\xdist,\ydist) to[out=90, in=315] (0,2*\ydist);
      \node (v) at (-2*\xdist,\ydist / 2)
{$\rt{\colt}$};
   \end{tikzpicture} \quad = \quad
   \begin{tikzpicture}[baseline=(current  bounding  box.center), 
                       inner sep=0mm, outer sep=0mm]
      \node (v) at (-1.5*\xdist,\ydist / 2)
{$\partial_{\cols}\rt{\colt}$};
      \draw[sline] (0, -2*\ydist) to (0,3*\ydist);
   \end{tikzpicture}
\]

This shows that if $p=2$, then $\pre{k}{B}_s \pre{k}{B}_t \pre{k}{B}_s$ does not decompose
as $\pre{k}{B}_{sts} \oplus \pre{k}{B}_s$, but remains indecomposable. %Thus we have found 
%$2$-torsion in the intersection cohomology of the Schubert variety associated to 
%$sts$. 
We get
\[ \pcan{sts} = 
   \begin{cases}
      \kl{sts} + \kl{s} & \text{if } p = 2 \text{,} \\
      \kl{sts} & \text{otherwise.}
   \end{cases}
\]

Swapping the roles of $s$ and $t$, the same calculation yields
\[ \pcan{tst} = \kl{tst} \]
for all primes $p$ as $\langle \rt{s}, \cort{t} \rangle = -1$.

Observe that in this case the whole local intersection form of $(s,t,s)$ at $s$ can 
be calculated using the formula in the nil Hecke ring which we explained in
\Cref{secNilHeck}.

\subsection{Type \texorpdfstring{$G_2$}{G2}}
We label the simple reflections as follows:
\[\begin{tikzpicture}[auto, baseline=(current  bounding  box.center)]
      \draw (0,2*\edgeShift) -- (-1,2*\edgeShift);
      \draw (0,0) -- (-1,0);
      \draw (0,-2*\edgeShift) -- (-1,-2*\edgeShift);
      \path (0,0) to node[Greater] (mid) {} (-1,0);
      \draw (mid.center) to +(30:\wingLen);
      \draw (mid.center) to +(330:\wingLen);
      \node [DynNode] (1) at (-1,0) {$s$};
      \node [DynNode] (2) at (0,0) {$t$};
   \end{tikzpicture} \]
That means for the pairing between the simple coroots and roots:
\begin{align*}
   \langle \rt{t}, \cort{s} \rangle &= -3 \text{,}\\
   \langle \rt{s}, \cort{t} \rangle &= -1 \text{.}
\end{align*}

For all primes $p > 3$ the Kazhdan-Lusztig basis coincides with the $p$-canonical
basis. Since the Cartan matrix is symmetric modulo $2$, the $2$-canonical basis
is stable under swapping $s$ and $t$. Here is a summary of the results for 
$p \in \{2, 3\}$:\\
% \begin{center}
%    \begin{tabular}{| p{\textwidth /2 - 0.5cm} | p{\textwidth / 2 - 0.5cm} |}
%       \hline \bigstrut
%       {$\begin{aligned}
%          \pcan[2]{stst} &= \kl{stst} + \kl{st} \\
%          \pcan[2]{ststs} &= \kl{ststs} + \kl{s} \\[0.5cm]
%          \pcan[2]{x} &= \kl{x} \quad \text{for } \\
%          x &\nin \{stst, tsts, ststs, tstst\}
%       \end{aligned}$} & \bigstrut
%       {$\begin{aligned}
%          \pcan[3]{sts} &= \kl{sts} + \kl{s} \\
%          \pcan[3]{ststs} &= \kl{ststs} + \kl{sts} \\[0.5cm]
%          \pcan[3]{x} &= \kl{x} \quad \text{for } \\
%          x &\nin \{sts, ststs\}
%       \end{aligned}$} \\ \hline
%    \end{tabular} 
% \end{center}
\begin{align*}
   \pcan[2]{stst} &= \kl{stst} + \kl{st}  & \pcan[3]{sts} &= \kl{sts} + \kl{s} \\
   \pcan[2]{ststs} &= \kl{ststs} + \kl{s} & \pcan[3]{ststs} &= \kl{ststs} + \kl{sts} \\[0.5cm]
   \pcan[2]{x} &= \kl{x} \quad \text{for } & \pcan[3]{x} &= \kl{x} \quad \text{for } \\
   x &\nin \{stst, tsts, ststs, tstst\} & x &\nin \{sts, ststs\}
\end{align*}

In this example, all the calculations needed to determine $\pcan{x}$ for 
$x \notin \{ ststs, tstst \}$ can be carried out using the formula in the
nil Hecke ring from \Cref{secNilHeck}. For $(s,t,s,t,s)$ (resp. $(t,s,t,s,t)$) there
is a subexpression of defect $0$ expressing $s$ (resp. $t$) that contains a $D1$:
\[ (U1, U1, U0, D1, D0) \text{.} \]

We will illustrate how useful the formula in the nil Hecke ring is by calculating
the intersection form of $(s,t,s,t)$ at $st$. For the Kazhdan-Lusztig basis we know:
\[ \kl{s} \kl{t} \kl{s} \kl{t} = \kl{stst} + 2 \kl{st} \text{.} \]
There are two subexpressions of $(s,t,s,t)$ of defect $0$ expressing $st$:
\begin{align*}
   \expr{e}^1 \defeq &(U1, U0, D0, U1) \\
   \expr{e}^2 \defeq &(U1, U1, U0, D0) 
\end{align*}
We need to calculate the coefficient of $D_{st}$ in the following
elements of the nil Hecke ring:
\begin{align*}
   d(\expr{e}^1, \expr{e}^1): \quad &D_s \rt{t} D_s D_t = \partial_s(\rt{t}) D_{st}\\
   d(\expr{e}^1, \expr{e}^2): \quad &D_s 1 1 D_t =D_{st} \\
   d(\expr{e}^2, \expr{e}^2): \quad &D_s D_t \rt{s} D_t = \partial_t(\rt{s}) D_{st} 
\end{align*}
Therefore the local intersection form of $stst$ at $st$ is given by
\[
\left ( 
   \begin{matrix} 
      -3 & 1 \\
      1 & 1
   \end{matrix} \right )
\]
which implies the result stated above.

\subsection{Type \texorpdfstring{$\tilde{A}_1$}{A1~}}
We label the simple reflections in $S_a$ as follows:
\[\begin{tikzpicture}[auto, baseline=(current  bounding  box.center)]
      \draw (0,0) to node[labeling, label=above:$\infty$] {} (-1,0);
      \node [DynNode] (1) at (-1,0) {$s$};
      \node [DynNode] (2) at (0,0) {$t$};
   \end{tikzpicture} \]
That means for the pairing between the simple coroots and roots:
\begin{align*}
   \langle \rt{t}, \cort{s} \rangle &= -2 \text{,}\\
   \langle \rt{s}, \cort{t} \rangle &= -2 \text{.}
\end{align*}

As we discussed in \Cref{secGeomSatake}, knowledge of the characters of the
tilting modules for $SL_2$ in characteristic $p$ gives us part of the $p$-canonical
basis for the Weyl group of type $\widetilde{A}_1$. Actually in this case
two miracles occur:
\begin{enumerate}
   \item Donkin's tilting tensor product theorem (\cite[Proposition 2.1]{DoTiltingMods}) together
         with the knowledge of the characters of the fundamental tilting modules
         (e.g. see \cite[Lemma 1.1]{DH}) allow us to determine the characters of
         all indecomposable tilting modules. (This is the only semi-simple group 
         for which all tilting characters are known).
   \item The set $\{ w_{\lambda} \; \vert \; \lambda \in \cocharlat_+ \}$ together
         with its image under the automorphism $s \leftrightarrow t$ yields the 
         set $W_a \setminus \{ \id \}$. In particular, knowledge of the tilting 
         characters gives the whole $p$-canonical basis for $SL_2$.
\end{enumerate}

To bypass the calculations, the following result (see \cite[Lemma 6]{EH}) is useful
because it gives a combinatorial description of the $\Delta$-multiplicities of an
indecomposable tilting module.

\begin{lem}
   Let $n \in \Z_{\geqslant 0}$. Write $n$ uniquely as $\sum_{i = 0}^l n_i p^i$ 
   with $p-1 \leqslant n_i \leqslant 2p-2$ for $i < l$ and $0 \leqslant n_l 
   \leqslant p-1$.
   
   Then $\Delta(m)$ occurs in a $\Delta$-flag of $T(n)$ if and only if 
   $m = \sum_{i = 0}^l m_i p^i$ where $m_j = n_j$ or $m_j = 2p - 2 - n_j$ for
   $j < l$ and $m_l = n_l$. Moreover, the multiplicity $[T(n) : \Delta(m)]$ is
   at most one.
\end{lem}

Observe that in the last lemma for $m = \sum_{i = 0}^l m_i p^i$ the coefficients
are not required to satisfy $p-1 \leqslant m_i \leqslant 2p-2$ for 
$i < m$. For $n = \sum_{i = 0}^l n_i p^i$ as in the lemma there
are precisely $2^r$ natural numbers $m$ such that $\Delta(m)$ occurs
in a $\Delta$-flag of $T(n)$ where $r = \vert \{ 0 \leqslant i < l \; \vert \; 
n_i \neq p-1 \}$. 

To give an example of how this works, consider $n=15$. This is
uniquely written as $15 = 3\cdot 3^0 + 4 \cdot 3 + 0 \cdot 3^2$ (as described in
the last lemma; the last $0$ matters!) and thus $4$ standard modules occur 
(with multiplicity one) in a $\Delta$-flag of $T(15)$, namely $\Delta(m)$ for $m$ among the 
following:
\begin{align*}
   1 &= 1 \cdot 3^0 + 0 \cdot 3^1 + 0 \cdot 3^2\text{,} \\
   3 &= 3 \cdot 3^0 + 0 \cdot 3^1 + 0 \cdot 3^2\text{,} \\
   13 &= 1 \cdot 3^0 + 4 \cdot 3^1 + 0 \cdot 3^2\text{,} \\
   15 &= 3 \cdot 3^0 + 4 \cdot 3^1 + 0 \cdot 3^2\text{.} \\
\end{align*}
See \Cref{figMultTilt} at the end of the paper for the multiplicities 
of $\Delta(m)$ in $T(n)$ for $p = 3$ where each black box represents a one. 
Using \Cref{lemPCanTilt} we get for example:
\begin{alignat*}{9}
   &\pcan[3]{s}        &&= \kl{s} && && && && && && &&\\
   &\pcan[3]{st}       &&=        &&\kl{st} && && && && && &&\\
   &\pcan[3]{sts}      &&=        &&        &&\kl{sts} && && && && &&\\
   &\pcan[3]{stst}     &&=        &&\kl{st} &&\ \  +   &&\kl{stst} && && && &&\\
   &\pcan[3]{ststs}    &&= \kl{s} &&        &&\ \  +   &&          &&\kl{ststs} && && &&\\
   &\pcan[3]{ststst}   &&=        &&        &&         &&          &&           &&\kl{ststst} && &&\\
   &\pcan[3]{stststs}  &&=        &&        &&         &&          &&\kl{ststs} &&\ \ \ +       &&\kl{stststs} &&\\
   &\pcan[3]{stststst} &&=        &&        &&         &&\kl{stst} &&           &&\ \ \ +       &&             &&\kl{stststst} \\
\end{alignat*}

\def \xspace{0.5}
\def \yspace{0.5}
\def \num{41}

\begin{figure}[p]
   \centering
   \resizebox{\textwidth}{!}{
   \begin{tikzpicture}[auto, baseline=(current  bounding  box.center)]
      \pgfmathsetmacro{\max}{\num-1};
      \foreach \i in {0, 1, ..., \max} {
         \pgfmathsetmacro{\x}{\i*\xspace + \xspace/2};
         %\pgfmathsetmacro{\y}[-\i*\yspace - \yspace/2};
         \node[labeling, label=above:\footnotesize$\i$] (\i) at (\x,0) {};
         \node[labeling, label=left:\footnotesize$\i$] (T\i) at (0,-\x) {};
         \draw[color=black!20] (\i*\xspace,\yspace) to (\i*\xspace, -\num*\yspace);
         \draw[color=black!20] (-\xspace, -\i*\yspace) to (\num*\xspace,-\i*\yspace);
      }
      
      \foreach \x/\y in {1/3, 0/4, 4/6, 3/7, 7/9, 6/10, 5/11, 4/12, 6/12, 3/13, 7/13, 
                         2/14, 1/15, 3/15, 0/16, 4/16, 10/12, 9/13, 13/15, 12/16,
                         16/18, 15/19, 14/20, 13/21,15/21, 12/22, 16/22, 11/23, 10/24, 12/24, 9/25, 13/25,
                         19/21, 18/22, 22/24, 21/25, 25/27, 24/28, 23/29, 22/30, 24/30, 21/31, 25/31,
                         20/32, 19/33, 21/33, 18/34, 22/34, 17/35, 16/36, 18/36, 15/37, 19/37, 14/38, 20/38, 13/39, 15/39, 19/39, 21/39,
                         28/30, 27/31, 31/33, 30/34, 34/36, 33/37, 32/38, 31/39, 33/39, 37/39,
                         12/40, 16/40, 18/40, 22/40, 30/40, 34/40, 36/40} {
         \pgfmathsetmacro{\Px}{\x*\xspace}
         \pgfmathsetmacro{\Py}{-\y*\yspace}
         \pgfmathsetmacro{\Qx}{(\x + 1)*\xspace}
         \pgfmathsetmacro{\Qy}{-(\y + 1)*\yspace}
         \draw[fill=black] (\Px, \Py) rectangle (\Qx, \Qy);
      }
      
      \foreach \i in {0, 1, ..., \max} {
         \pgfmathsetmacro{\Px}{\i*\xspace}
         \pgfmathsetmacro{\Py}{-\i*\yspace}
         \pgfmathsetmacro{\Qx}{(\i + 1)*\xspace}
         \pgfmathsetmacro{\Qy}{-(\i + 1)*\yspace}
         \draw[fill=black] (\Px, \Py) rectangle (\Qx, \Qy);
      }
      \draw (-\xspace,0) to (\num*\xspace,0);
      \draw (0,\yspace) to (0, -\num*\yspace);
      \draw (0,0) to (-\xspace, \yspace);
      \node[labeling, label={center:\footnotesize$n$}] (n) at (-\xspace, 0.2cm) {};
      \node[labeling, label={center:\footnotesize$m$}] (m) at (-0.2cm, \yspace) {};
   \end{tikzpicture}}
   \caption{The multiplicities of $\Delta(m)$ in $T(n)$ for $p=3$.}
   \label{figMultTilt}
\end{figure}
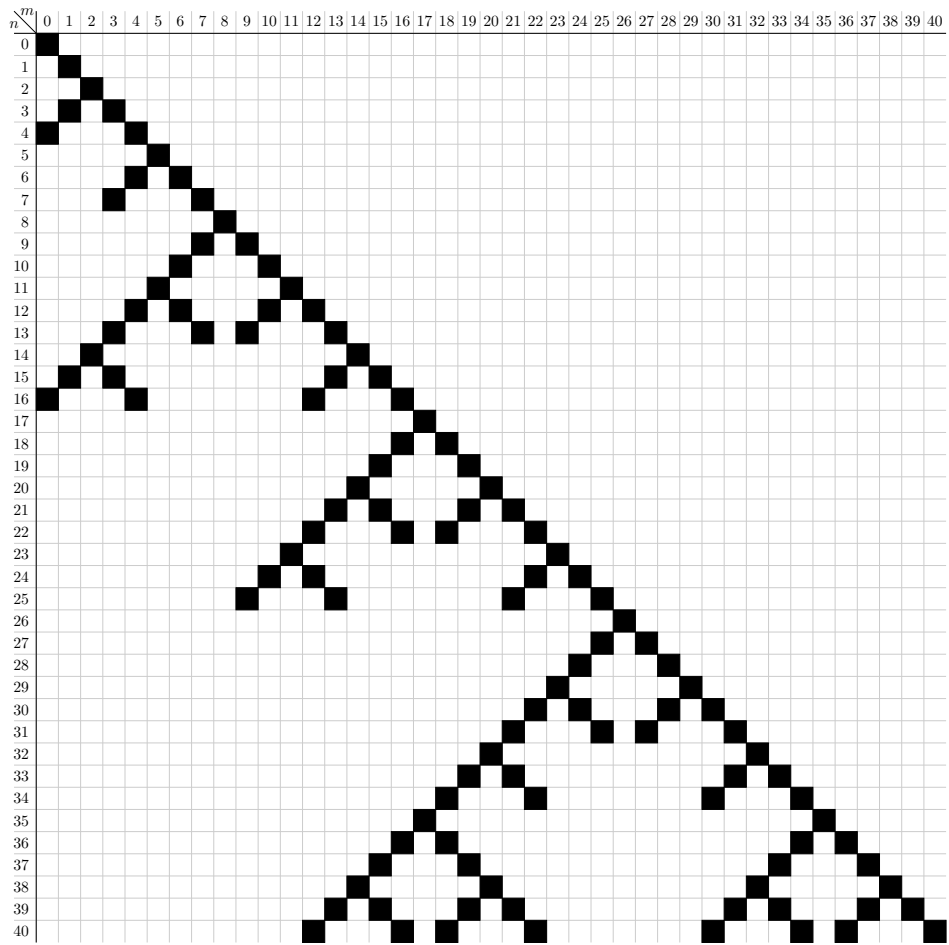

Note that even in this relatively simple example one sees a beautiful fractal-like
structure emerging!

\subsection{Types \texorpdfstring{$B_3$}{B3} and \texorpdfstring{$C_3$}{C3}}
In the Dynkin diagrams of types $B_3$ and $C_3$ we label the simple
reflections as follows:
\begin{gather*}
   B_3 :
   \begin{array}{c}
   \begin{tikzpicture}[auto, baseline=(current  bounding  box.center)]
      \draw (1,0) -- (0,0);
      \draw (0,\edgeShift) -- (-1,\edgeShift);
      \draw (0,-\edgeShift) -- (-1,-\edgeShift);
      \path (0,0) to node[Greater] (mid) {} (-1,0);
      \draw (mid.center) to +(30:\wingLen);
      \draw (mid.center) to +(330:\wingLen);
      \node [DynNode] (1) at (-1,0) {$1$};
      \node [DynNode] (2) at (0,0) {$2$};
      \node [DynNode] (3) at (1,0) {$3$};
   \end{tikzpicture} \end{array} \\
   C_3: 
   \begin{array}{c} 
      \begin{tikzpicture}[auto, baseline=(current  bounding  box.center)]
         \draw (1,0) -- (0,0);
         \draw (0,\edgeShift) -- (-1,\edgeShift);
         \draw (0,-\edgeShift) -- (-1,-\edgeShift);
         \path (-1,0) to node[Greater] (mid) {} (0,0);
         \draw (mid.center) to +(150:\wingLen);
         \draw (mid.center) to +(210:\wingLen);
         \node [DynNode] (1) at (-1,0) {$1$};
         \node [DynNode] (2) at (0,0) {$2$};
         \node [DynNode] (3) at (1,0) {$3$};
      \end{tikzpicture}
   \end{array}
\end{gather*}

The only interesting case is $p=2$. The following table 
gives an overview over all the Weyl group elements for which the $2$-canonical
basis differs from the Kazhdan-Lusztig basis. It illustrates the dependence
of the $2$-canonical basis on the type of the root system. Even though 
the combinatorics in types $B_3$ and $C_3$ are the same, the corresponding 
$2$-canonical bases are quite different.\\

\vspace*{0.2cm}
\noindent
\resizebox{\linewidth}{!}{
   \begin{tabular}{|l|c|c|}
      \hline
      & $B_3$ & $C_3$ \T \B \\
      \hline
      $\pcan[2]{212}$ & $\kl{212}$ & $\kl{212} + \kl{2}$ \T \\ 
      $\pcan[2]{121}$ & $\kl{121} + \kl{1}$ & $\kl{121}$ \\ 
      $\pcan[2]{3212}$ & $\kl{3212}$ & $\kl{3212} + \kl{32}$ \\ 
      $\pcan[2]{2123}$ & $\kl{2123}$ & $\kl{2123} + \kl{23}$ \\ 
      $\pcan[2]{1321}$ & $\kl{1321} + \kl{13}$ & $\kl{1321}$ \\ 
      $\pcan[2]{1213}$ & $\kl{1213} + \kl{13}$ & $\kl{1213}$ \\ 
      $\pcan[2]{32123}$ & $\kl{32123}$ & $\kl{32123} + \kl{232} + \kl{3}$ \\ 
      $\pcan[2]{21232}$ & $\kl{21232}$ & $\kl{21232} + \kl{232}$ \\ 
      $\pcan[2]{23212}$ & $\kl{23212}$ & $\kl{23212} + \kl{232}$ \\ 
      $\pcan[2]{21321}$ & $\kl{21321} + \kl{213}$ & $\kl{21321}$ \\ 
      $\pcan[2]{12132}$ & $\kl{12132} + \kl{132}$ & $\kl{12132}$ \\ 
      $\pcan[2]{232123}$ & $\kl{232123}$ & $\kl{232123} + (v+v^{-1}) \kl{232}$ \\ 
      $\pcan[2]{212321}$ & $\kl{212321}$ & $\kl{212321} + \kl{2321}$ \\ 
      $\pcan[2]{121321}$ & $\kl{121321} + \kl{1212} + \kl{1321} +  \kl{1213} + \kl{13}$ 
                         & $\kl{121321}$ \\ 
      $\pcan[2]{123212}$ & $\kl{123212}$ & $\kl{123212} + \kl{1232}$ \\ 
      $\pcan[2]{2123212}$ & $\kl{2123212}$ 
                          & $\kl{2123212} + \kl{21232} + \kl{23212} + \kl{232}$ \\ 
      $\pcan[2]{1212321}$ & $\kl{1212321} + \kl{12123}$ & $\kl{1212321}$ \\ 
      $\pcan[2]{1213212}$ & $\kl{1213212} + \kl{13212}$ & $\kl{1213212}$ \\ 
      $\pcan[2]{21232123}$ & $\kl{21232123}$ & $\kl{21232123} + \kl{232123}$ \\ 
      $\pcan[2]{12123212}$ & $\kl{12123212} + \kl{1212}$ & $\kl{12123212}$ \\ 
      $\pcan[2]{12132123}$ & $\kl{12132123} + \kl{132123}$ & $\kl{12132123}$ \B \\
      \hline
   \end{tabular}}
\vspace*{0.2cm}
   
The most interesting entry in the whole table occurs for type $C_3$ and the
element $232123 \in W$ where we have
\[ \pcan[2]{232123} = \kl{232123} + (v + v^{-1}) \kl{232} \]

This means that, in the decomposition of $B_{232123} \otimes_{\rO} \K$ into
indecomposable objects in $\pre{\K}{\HC}$, non-self-dual summands (i.e.
with a non-trivial grading shift) occur.

\subsection{Type \texorpdfstring{$D_4$}{D4}}
We label the simple reflections as follows:
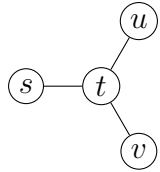
\[ \begin{tikzpicture}[auto, baseline=(current  bounding  box.center)]
      \node [DynNode] (t) at (0,0) {$t$};
      \node [DynNode] (s) at (180:1) {$s$};
      \node [DynNode] (u) at (60:1) {$u$};
      \node [DynNode] (v) at (-60:1) {$v$};

      \draw (t) to (s);
      \draw (t) to (u);
      \draw (t) to (v);
   \end{tikzpicture} \]

It turns out that the $p$-canonical basis and the Kazhdan-Lusztig basis coincide 
for all primes except for $p = 2$. There are four elements $x \in W$ with 
$\pcan[2]{x} \ne \kl{x}$. If $x = suvtsuv$ then we have
\[
   \pcan[2]{t_1xt_2} = \kl{t_1xt_2} + \kl{t_1suvt_2}
\]
for $t_1, t_2 \in \langle t \rangle$. We will give some more details on how to
calculate $\pcan[2]{x}$. We start out by decomposing the corresponding Bott-Samelson
object into indecomposable objects in $\pre{\K}{\HC}$ to get
\[ \std{s} \std{u} \std{v} \std{t} \std{s} \std{u} \std{v}
   = \kl{x} + (v^{-2} + 3 + v^2) \kl{suv} \text{.}
\]
As subexpressions of $(s,u,v,t,s,u,v)$ expressing $suv$ we get
\begin{alignat*}{2}
   &(U1,U1,U1,U0,D0,D0,D0) \quad && \text{of defect } -2 \text{,} \\
   &(U1,U1,U0,U0,D0,D0,U1) && \text{of defect } 0 \text{,} \\
   &(U1,U0,U1,U0,D0,U1,D0) && \text{of defect } 0 \text{,} \\
   &(U0,U1,U1,U0,U1,D0,D0) && \text{of defect } 0 \text{,} \\
   &(U1,U0,U0,U0,D0,U1,U1) && \text{of defect } 2 \text{,} \\
   &(U0,U1,U0,U0,U1,D0,U1) && \text{of defect } 2 \text{,} \\
   &(U0,U0,U1,U0,U1,U1,D0) && \text{of defect } 2 \text{,} \\
   &(U0,U0,U0,U0,U1,U1,U1) && \text{of defect } 4 \text{.}
\end{alignat*}

The light leaf morphism of degree $-2$ pairs with the three light leaf morphisms
of degree $2$ to give the matrix
\[
\left ( 
   \begin{matrix} 
      -1 \\ 
      -1 \\ 
      -1 

   \end{matrix} \right ) \text{.}
\]

The light leaf morphisms corresponding to the subexpressions of defect $0$
are the following:
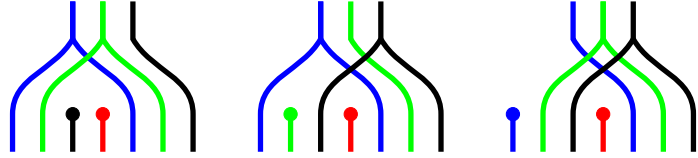
\[
   \begin{tikzpicture}[auto]
      \draw[sline] (\xdist,0) -- (\xdist,\ydist) to[out=90, in=300]  (-\xdist,3*\ydist) -- (-\xdist,4*\ydist);

      \draw[sline] (-3*\xdist,0) -- (-3*\xdist,\ydist) to[out=90, in=240] (-\xdist,3*\ydist) -- (-\xdist,4*\ydist);

      \draw[uline] (2*\xdist,0) -- (2*\xdist,\ydist) to[out=90, in=300]  (0,3*\ydist) -- (0,4*\ydist);

      \draw[uline] (-2*\xdist,0) -- (-2*\xdist,\ydist) to[out=90, in=240] (0,3*\ydist) -- (0,4*\ydist);

      \draw[line=black] (-\xdist,0) -- (-\xdist,\ydist);

      \draw[dot=black] (-\xdist,\ydist) circle (\circSize);

      \draw[line=black] (3*\xdist,0) -- (3*\xdist,\ydist) to[out=90, in=300] (\xdist,3*\ydist) -- (\xdist,4*\ydist);

      \draw[tline] (0,0) -- (0,\ydist);

      \draw[tdot] (0,\ydist) circle (\circSize);

   \end{tikzpicture} \qquad
   \begin{tikzpicture}[auto]
      \draw[sline] (\xdist,0) -- (\xdist,\ydist) to[out=90, in=300]  (-\xdist,3*\ydist) -- (-\xdist,4*\ydist);

      \draw[sline] (-3*\xdist,0) -- (-3*\xdist,\ydist) to[out=90, in=240] (-\xdist,3*\ydist) -- (-\xdist,4*\ydist);

      \draw[uline] (2*\xdist,0) -- (2*\xdist,\ydist) to[out=90, in=300]  (0,3*\ydist) -- (0,4*\ydist);

      \draw[uline] (-2*\xdist,0) -- (-2*\xdist,\ydist);
      \draw[udot] (-2*\xdist,\ydist) circle (\circSize);
      \draw[line=black] (-\xdist,0) -- (-\xdist,\ydist) to[out=90, in=240] (\xdist, 3*\ydist);
      \draw[line=black] (3*\xdist,0) -- (3*\xdist,\ydist) to[out=90, in=300] (\xdist,3*\ydist) -- (\xdist,4*\ydist);

      \draw[tline] (0,0) -- (0,\ydist);

      \draw[tdot] (0,\ydist) circle (\circSize);

   \end{tikzpicture} \qquad
   \begin{tikzpicture}[auto]
      \draw[sline] (\xdist,0) -- (\xdist,\ydist) to[out=90, in=300]  (-\xdist,3*\ydist) -- (-\xdist,4*\ydist);

      \draw[sline] (-3*\xdist,0) -- (-3*\xdist,\ydist);
      \draw[sdot] (-3*\xdist,\ydist) circle (\circSize);
      \draw[uline] (2*\xdist,0) -- (2*\xdist,\ydist) to[out=90, in=300]  (0,3*\ydist) -- (0,4*\ydist);

      \draw[uline] (-2*\xdist,0) -- (-2*\xdist,\ydist) to[out=90, in=240] (0,3*\ydist) -- (0,4*\ydist);
      \draw[line=black] (-\xdist,0) -- (-\xdist,\ydist) to[out=90, in=240] (\xdist, 3*\ydist);
      \draw[line=black] (3*\xdist,0) -- (3*\xdist,\ydist) to[out=90, in=300] (\xdist,3*\ydist) -- (\xdist,4*\ydist);

      \draw[tline] (0,0) -- (0,\ydist);

      \draw[tdot] (0,\ydist) circle (\circSize);

   \end{tikzpicture}
\]

Pairing them gives the following degree $0$ piece of the intersection form:
\[
\left( 
   \begin{matrix} 
      0 & -1 & -1 \\ 
      -1 & 0 & -1 \\ 
      -1 & -1 & 0

   \end{matrix} \right) \text{.}
\]
Note that the determinant of this matrix is $-2$ and its rank in characteristic $2$
is $2$. Therefore $\pre{k}B_{suvtsuv} \otimes_{\rO} \K$ decomposes as
\[ \pre{\K}B_{suvtsuv} \oplus \pre{\K}B_{suv} \]
giving the result we stated above. The geometry of this example is discussed
in the appendix of \cite{BW}. Note that all calculations presented in this section
can also be carried out using the formula in the nil Hecke ring 
(see \cite[\S 6.2]{HW}).

\subsection{Type \texorpdfstring{$A_n$}{An}}

According to \cite{BW}, the $p$-canonical basis and the Kazhdan-Lusztig basis
coincide for all primes $p$ for $n < 7$. Thus, we will describe the case $n = 7$
where the situation is quite remarkable.

For all primes $p \neq 2$ the $p$-canonical basis and the Kazhdan-Lusztig basis 
agree. We have $\pcan[2]{x} \ne \kl{x}$ for exactly 38 out of 40320
elements in $S_8$ and these examples fall into four classes. 

In the following we will denote for a subset $I \subseteq S$ the corresponding 
parabolic subgroup by $W_I = \langle s \in I \rangle \subseteq W$. If $W_I$ is 
finite, its longest element will be denoted by $w_I$. A permutation $\phi \in S_8$
will be displayed as a string $\phi(1) \phi(2) \dots \phi(8)$.

\emph{The Kashiwara-Saito singularity} (\cite{KaSa}): This corresponds to the
permutation $w = 62845173.$ We have
\[
   \pcan[2]{w} = \kl{w} + \kl{w_I}
\]
with $I = \{1,3,4,5,7\}$. There is a cluster of $16$ elements around the 
Kashiwara-Saito singularity described as follows. If we let $J = \{ 2,6 \}$, 
then we have
\[
   \pcan[2]{uwv} = \kl{uwv} + \kl{uw_Iv}.
\]
for all $u, v \in W_J$ unless $u = v = w_J$ in which case
\[
   %\pcan[2]{uwv} = \kl{uwv} + \kl{uw_Iv} + \kl{w_K}
   \pcan[2]{w_J w w_J} = \kl{w_J w w_J} + \kl{w_J w_I w_J} + \kl{w_K}
\]
where $K = \{ 1,2,3,5,6,7\}$.

\emph{The Hexagon singularity} (Braden's example in the appendix of \cite{BW}): 
Consider the permutation $w = 46718235$. We have
\[
   \pcan[2]{w} = \kl{w} + \kl{w_I}
\]
with $I = \{ 2, 3, 5, 6\}$. In this case we get a cluster of size $4$. 
For any $u, v \in \langle s_4 \rangle$ we have
\[
   \pcan[2]{uwv} = \kl{uwv} + \kl{uw_Iv}
\]
unless $u = v = s_4$ in which case
\[
   \pcan[2]{s_4ws_4} = \kl{s_4ws_4} + \kl{s_4w_Is_4} + \kl{w_K}
\]
where $K = \{1,3,4,5,7\}$. 

For the Kashiwara-Saito singularity and the Hexagon singularity
the calculation of the the local intersection form of $\expr{w}$ at $w_I$ using
the formula in the nil Hecke ring can be found in \cite[\S6.1 and \S 6.2]{HW}.
In both cases one can find $\expr{w}$, a reduced expression for $w$,
such that this local intersection form is a $1 \times 1$ matrix.

\emph{The waterfall:} Consider the permutation $w_1 = 67283415$.
Then we have
\[
   \pcan[2]{w_1u} = \kl{w_1u} + \kl{w_Iu}.
\]
for $I = \{1,2,4,5,6\}$ and $w_J \ne u \in W_J$ with $J = \{ 3,7\}$.

Similarly, if we let $w_2 = 57813462$, then we have
\[
   \pcan[2]{vw_2} = \kl{vw_2} + \kl{vw_I'}
\]
where $I' = \{ 2,3,4,6,7\}$ and $w_{J'} \ne v \in W_J$ with $J = \{1,5\}$.

The situation becomes more complicated due to the fact that for $u = w_J$ and $v =
w_{J'}$ we have $w_1u = vw_2 =: w$. In this case we get
\[
   \pcan[2]{w} = \kl{w} + \kl{w_Iu} + \kl{vw_I'}.
\]

Note that, unlike the Kashiwara-Saito and hexagon permutations
discussed above, the clusters containing $w_1$ and $w_2$ are neither
swapped nor fixed by the automorphism $s_i \mapsto s_{8-i}$ and thus
by applying the graph automorphism one obtains another seven elements
for which $\pcan[2]{x} \ne \kl{x}$. Hence the two ``waterfall'' clusters 
contain $14$ elements in total.

\emph{The basket:} Consider the permutation
\[
w = 84627351.
\]
Then for all $u, v \in \langle s_4 \rangle$ one has
\[
{}^2 \kl{uwv} = \kl{uwv} + \kl{uw_Iv}
\]
where $I = \{1,2,3,5,6,7\}$.

%\bibliographystyle{amsalpha}
%\bibliography{pcan}

\def\cprime{$'$}
\providecommand{\bysame}{\leavevmode\hbox to3em{\hrulefill}\thinspace}

\end{document}